\documentclass{amsart}
\usepackage{amsmath, amssymb}
\usepackage{cleveref}
\usepackage{graphicx}

\def\l{\left}
\def\r{\right}

\usepackage[margin=1.25in]{geometry}

\def\mR{\mathbb{R}}
\def\mRd{\mathbb{R}^d}
\def\mRdp{\mathbb{R}^{d+1}}
\def\vp{\varphi}
\def\div{\text{div}}
\def\PV{\text{P.V.}}
\def\Gts{\widetilde{G}_s}
\def\Th{\mathcal{T}_h}

\def\Nhc{\mathcal{N}_h^c}
\def\interp{\mathcal I_h}

\def\Cl{\Pi_h}
\def\ve{\varepsilon}
\def\pO{\partial\Omega}
\def\Oc{\overline{\Omega}}

\def\x{\texttt{x}}

\newcommand{\w}{\omega}

\newcommand{\eps}{\varepsilon}
\def\dist{\textrm{dist}}
\def\oT{T}
\def\oTp{T'}
\def\dty{d\widetilde{y}}

\newtheorem{Theorem}{Theorem}[section]
\newtheorem{Lemma}[Theorem]{Lemma}
\newtheorem{Proposition}[Theorem]{Proposition}
\newtheorem{Corollary}[Theorem]{Corollary}
\newtheorem{Remark}[Theorem]{Remark}
\newtheorem{Definition}[Theorem]{Definition}

\usepackage{algpseudocode,algorithm,algorithmicx}

\algrenewcommand\algorithmicrequire{\textbf{Input:}}
\algrenewcommand\algorithmicensure{\textbf{Output:}}

\numberwithin{equation}{section}

\begin{document}

\title[FE for nonlocal minimal graphs]{Finite element discretizations of nonlocal minimal graphs: convergence}

\author[J.P.~Borthagaray]{Juan Pablo~Borthagaray}
\address[J.P.~Borthagaray]{Departamento de Matem\'{a}tica y Estad\'{\i}stica del Litoral, Universidad de la Rep\'{u}blica, Salto, Uruguay}
\email{jpborthagaray@unorte.edu.uy}
\thanks{JPB has been supported in part by NSF grant DMS-1411808.}

\author[W.~Li]{Wenbo~Li}
\address[W.~Li]{Department of Mathematics, The University of Tennessee, Knoxville, USA}
\email{wli50@utk.edu}
\thanks{WL has been supported in part by NSF grant DMS-1411808 and the Patrick and
	Marguerite Sung Fellowship in Mathematics.}

\author[R.H.~Nochetto]{Ricardo H.~Nochetto}
\address[R.H.~Nochetto]{Department of Mathematics and Institute for Physical Science and Technology, University of Maryland, College Park, USA}
\email{rhn@umd.edu}
\thanks{RHN has been supported in part by NSF grant DMS-1411808}

%\date{Draft version of \today}

                    %%  Finite elements, Rayleigh-Ritz and Galerkin methods, finite methods;


\maketitle

%%%%%%%%%%%%%%%%%%%%%%%%%%%%
%%%%%%%%%%%%%%%%%%%%%%%%%%%%

\begin{abstract}
In this paper, we propose and analyze a finite element discretization for the computation of fractional minimal graphs of order~$s \in (0,1/2)$ on a bounded domain $\Omega$. Such a Plateau problem of order $s$ can be reinterpreted as a Dirichlet problem for a nonlocal, nonlinear, degenerate operator of order $s + 1/2$. We prove that our numerical scheme converges in $W^{2r}_1(\Omega)$ for all $r<s$, where $W^{2s}_1(\Omega)$ is closely related to the natural energy space. Moreover, we introduce a geometric notion of error that, for any pair of $H^1$ functions, in the limit $s \to 1/2$ recovers a weighted $L^2$-discrepancy between the normal vectors to their graphs. We derive error bounds with respect to this novel geometric quantity as well. In spite of performing approximations with continuous, piecewise linear, Lagrangian finite elements, the so-called {\em stickiness} phenomenon becomes apparent in the numerical experiments we present.
\end{abstract}

%%%%%%%%%%%%%%%%%%%%%%%%%%%%
%%%%%%%%%%%%%%%%%%%%%%%%%%%%
\section{Introduction} \label{sec:intro}
Several complex phenomena, such as those involving surface tension, can be interpreted in terms of perimeters. In general, perimeters provide a good local description of these intrinsically nonlocal phenomena. The study of fractional minimal surfaces, which can be interpreted as a non-infinitesimal version of classical minimal surfaces, began with the seminal works by Imbert \cite{Imbert09} and Caffarelli, Roquejoffre and Savin \cite{CaRoSa10}. 

As a motivation for the notion of fractional minimal sets let us show how it arises in the study of   a nonlocal version of the Ginzburg-Landau energy, extending a well-known result for classical minimal sets \cite{Modica87, ModicaMortola77}. Let $\Omega \subset \mRd$ be a bounded set with Lipschitz boundary, $\varepsilon > 0$ and define the energy
\[
\mathcal{J}_{\ve}[u;\Omega] = \frac{\ve^2}{2} \int_\Omega |\nabla u(x)|^2 \; dx + \int_{\Omega} W(u(x)) \; dx,
\]
where $W(t) = \frac14(1-t^2)^2$ is a double-well potential. Then, for every sequence $\{ u_\eps \}$ of minimizers of the rescaled functional $\mathcal{F}_{\ve}[u;\Omega]=\eps^{-1}\mathcal{J}_{\ve}[u;\Omega]$ with uniformly bounded energies there exists a subsequence $\{ u_{\eps_k} \}$ such that
\[ u_{\eps_k} \to \chi_E - \chi_{E^c} \quad \mbox{in } L^1(\Omega),
\]
where $E$ is a set with minimal perimeter in $\Omega$.
Analogously, given $s \in (0,1)$, consider the energy
\[
\mathcal{J}^s_{\ve}[u;\Omega] = \frac{\ve^{2s}}{2}\iint_{Q_{\Omega}} \frac{|u(x) - u(y)|^2}{|x-y|^{n+2s}} \; dx dy + \int_{\Omega} W(u(x)) \; dx,
\]
where ${Q_{\Omega} = \l( \mRd \times \mRd \r) \setminus \l( {\Omega}^c \times {\Omega}^c \r)}$,
and rescale it as
\[
\mathcal{F}^s_\ve[u;\Omega] =
\left\lbrace \begin{array}{ll}
\ve^{-2s} \mathcal{J}^s_{\ve}[u;\Omega] & \mbox{if } s \in (0,1/2); \\
\ve^{-1} | \log \ve|^{-1} \mathcal{J}^s_{\ve}[u;\Omega] & \mbox{if } s = 1/2; \\
\ve^{-1} \mathcal{J}^s_{\ve}[u;\Omega] & \mbox{if } s \in (1/2,1) .
\end{array}
\right.
\]
Note that the first term in the definition of $\mathcal{J}^s_\eps$ involves the $H^s(\mRd)$-norm of $u$, except that the interactions over $\Omega^c \times \Omega^c$ are removed; for a minimization problem in $\Omega$, these are indeed fixed. As proved in \cite{SaVa12Gamma}, for every sequence $\{ u_\eps \}$ of minimizers of $\mathcal{F}^s_{\ve}$ with uniformly bounded energies there exists a subsequence $\{ u_{\eps_k} \}$ such that
\[ u_{\eps_k} \to \chi_E - \chi_{E^c} \quad \mbox{in } L^1(\Omega) \quad \mbox{as } \ve_k \to 0^+.
\]
If $s \in [1/2,1)$, then $E$ has minimal classical perimeter in $\Omega$, whereas if $s \in (0,1/2)$, then $E$ minimizes the nonlocal $s$-perimeter functional given by \Cref{def:s-perimeter}.

Other applications of nonlocal perimeter functionals include motions of fronts by nonlocal mean curvature \cite{CaSo10,ChaMorPon12,ChaMorPon15} and nonlocal free boundary problems \cite{CaSaVa15,DipiSavinVald15,DipiVald18}. We also refer the reader to 
\cite[Chapter 6]{BucurValdinoci16} and \cite{CoFi17} for nice introductory expositions to the topic.

The goal of this work is to design and analyze finite element schemes in order to compute fractional minimal sets over cylinders $\Omega\times\mRd$ in $\mRdp$, provided the external data is a subgraph. In such a case, minimal sets turn out to be subgraphs in the interior of the domain $\Omega$ as well, and the minimization problem for minimal sets can be equivalently stated as a minimization problem for a functional acting on functions $u \colon \mRd \to \mR$, given by
\begin{equation}\label{Is}
I_s[u] = \iint_{Q_{\Omega}} F_s\l(\frac{u(x)-u(y)}{|x-y|}\r) \frac{1}{|x-y|^{d+2s-1}} \;dxdy,
\end{equation}
where $Q_\Omega = (\mRd\times\mRd) \setminus(\Omega^c\times\Omega^c)$, $\Omega^c=\mRd\setminus\Omega$ is the complement of $\Omega$ in $\mRd$ and $F_s:\mR\to\mR$ is a suitable convex nonnegative function. This is the $s$-fractional version of the classical graph area functional
\[
I [u] = \int_\Omega \sqrt{1 + |\nabla u (x)|^2 } \, dx 
\]
among suitable functions $u:\Omega\to\mR$ satisfying the Dirichlet condition $u=g$ on $\partial\Omega$. A crucial difference between the two problems is that the Dirichlet condition for $I_s[u]$ must be imposed in $\Omega^c$, namely
\[
u = g \quad\text{in }\Omega^c.
\]
We propose a discrete counterpart of $I_s[u]$ based on piecewise linear Lagrange
finite elements on shape-regular meshes, and prove a few properties of the discrete
solution $u_h$, including convergence in $W^{2r}_1(\Omega)$ for any $0<r<s$ as the
meshsize $h$ tends to $0$. We point out that $u \in W^{2s}_1(\Omega)$ is the minimal regularity needed to guarantee that $I_s[u]$ is finite. We also derive error estimates for a novel geometric quantity related to the concept of fractional normal.

A minimizer of $I_s[u]$ satisfies the variational equation
\begin{equation}\label{variation-s}
\iint_{Q_{\Omega}} \widetilde{G}_s\l(\frac{u(x)-u(y)}{|x-y|}\r)
\frac{\big(u(x)-u(y)\big)\big(v(x)-v(y)\big)}{|x-y|^{d+2s+1}} \;dxdy = 0
\end{equation}
for all functions $v\in W^{2s}_1(\Omega)$ such that $v=0$ in $\Omega^c$. Hereafter,
$\widetilde{G}_s(\rho):=\rho^{-1} F_s'(\rho)$ has the property that
$\widetilde{G}_s(\rho)\to0$
as $|\rho|\to\infty$, which makes the equation for $u$ both nonlinear and degenerate.
This extends to $0<s<1$ the equation
\begin{equation}\label{variation-1}
\int_\Omega \frac{\nabla u(x) \cdot \nabla v(x)}{\sqrt{1+|\nabla u(x)|^2}} \;dx = 0
\end{equation}
for minimal graphs. Moreover, this extends the quadratic case $F_s(\rho)=C_s\rho^2$,
which leads to the equation for the integral fractional Laplacian $(-\Delta)^s$ of order
$s+1/2<1$,
\[
\langle (-\Delta)^s u, v \rangle = 2 C_s \iint_{Q_{\Omega}} 
\frac{\big(u(x)-u(y)\big)\big(v(x)-v(y)\big)}{|x-y|^{d+2s+1}} \;dxdy = 0.
\]
Our nonlinear solver hinges on the linear solver for $(-\Delta)^s u$ of \cite{AcosBersBort2017short}. In fact, we develop a discrete gradient flow and a Newton method, which are further discussed in \cite{BoLiNo19computational} along with several numerical experiments that illustrate and explore the boundary behavior of $u$. In this paper we present simple numerical experiments.

Let us briefly review the literature on finite element discretizations of $(-\Delta)^s$ on bounded domains $\Omega$ in $\mRd$. We refer to \cite{AcosBersBort2017short, AcosBort2017fractional, AiGl18, BoLePa17, DEliaGunzburger} for homogeneous Dirichlet conditions $g=0$ in $\Omega^c$ as well as details on convergence of the schemes and their implementation.
On the other hand, methods have been proposed to deal with arbitrary nonhomogeneous Dirichlet conditions $g\ne0$ in $\Omega^c$, either based on weak imposition of the datum by using Lagrange multipliers \cite{AcBoHe18} or on the approximation of the Dirichlet problem by Robin exterior value problems \cite{AntilKhatriWarma18}. We refer to the survey \cite{BBNOS18} for additional discussion, comparison of methods, and references.
Moreover, the fractional obstacle problem for $(-\Delta)^s$ has been studied in \cite{BoLeSa18,BoNoSa18,BuGu18}, where regularity estimates and convergence rates are derived.

This paper seems to be the first to treat numerically fractional minimal graphs.  We now outline its contents and organization.
\Cref{sec:formulation} deals directly with the functional $I_s$, thereby avoiding a lengthly discussion of fractional perimeters, which is included in Appendix \ref{appendix:perimeter}.
\Cref{sec:formulation} studies some properties of $I_s$ and introduces the variational formulation \eqref{variation-s}.
The discrete formulation and the necessary tools for its analysis, such as localization of fractional order seminorms, quasi-interpolation operators and interpolation estimates are described in \Cref{sec:numerical-method}. 
In \Cref{sec:convergence} we show that our discrete energy is consistent. This leads to convergence of discrete solutions $u_h$ to $s$-minimal graphs in $W^{2r}_1(\Omega)$ for every $0< r < s$ as the largest element size $h$ tends to $0$ without any additional regularity of $u$ beyond $u\in W^{2s}_1(\Omega)$.

A more intrinsic error measure than the Sobolev norm in $W^{2r}_1(\Omega)$ exploits the geometric structure of $I_s$. For the classical case $s=1$, set $\widehat{\nu}(\pmb{a}) = \frac{(\pmb{a},-1)}{Q(\pmb{a})}$ with $Q(\pmb{a})=\sqrt{1+|\pmb{a}|^2}$ and consider the geometric error between two functions $u,v \colon \Omega \to \mR$
\[
e(u,v) = \l(\int_{\Omega} \ \Big| \widehat{\nu}(\nabla u) - \widehat{\nu}(\nabla v) \Big|^2 \;\frac{Q(\nabla u) + Q(\nabla u_h)}{2} dx \r)^{\frac12}.
\]
This quantity $e(u,v)$ gives a weighted $L^2$-estimate of the discrepancy between the unit normals to the graphs of $u$ and $v$ \cite{FierroVeeser03}. \Cref{sec:geometric-error} deals with a novel nonlocal geometric quantity $e_s(u,v)$ that mimics $e(u,v)$. We first derive an upper bound for $e_s(u,u_h)$, where $u$ is the $s$-minimal graph and $u_h$ is its discrete counterpart, without regularity of $u$ as well as error estimates under realistic regularity assumptions on $u$. We next prove that the nonlocal quantity $e_s(u,v)$ recovers its local counterpart $e(u,v)$ as $s \to \frac12^-$. In doing so, we also prove convergence of the forms \eqref{variation-s} to \eqref{variation-1} as $s \to \frac12^-$.

\Cref{sec:numerics} presents experiments that illustrate the performance of the proposed numerical methods and the behavior of $s$-minimal graphs. We conclude with some technical material in appendices. We collect definitions of and results for fractional perimeters in Appendix \ref{appendix:perimeter} and exploit them to derive the energy $I_s$ in Appendix \ref{appendix:energy}. \looseness=-1

%%%%%%%%%%%%%%%%%%%%%%%%%%%%
%%%%%%%%%%%%%%%%%%%%%%%%%%%%
\section{Formulation of the problem} \label{sec:formulation}
%%%%%%%%%%%%%%%%%%%%%%%%%%%%
%%%%%%%%%%%%%%%%%%%%%%%%%%%%
As a motivation for the formulation of the problem we are concerned with in this work, we first visit the classical minimal graph problem. Let $\Omega \subset \mRd$ be an open set with sufficiently smooth boundary, and let $g \colon \partial \Omega \to \mRd$ be given. Then, the Plateau problem consists in finding $u \colon \Omega \to \mRd$ that minimizes the graph surface area functional
\begin{equation} \label{E:MS-Energy-Graph}
I [u] = \int_\Omega \sqrt{1 + |\nabla u (x)|^2 } \, dx 
\end{equation}
among a suitable class of functions satisfying $u = g$ on $\partial \Omega$. For simplicity, let us assume that such a class is a subset of $H^1(\Omega)$. By taking first variation of $I$, it follows that the minimizer $u$ satisfies
\begin{equation} \label{eq:classical-MS}
\int_\Omega \frac{\nabla u(x) \cdot \nabla v (x)}{\sqrt{1 + |\nabla u (x)|^2 }} \, dx = 0 , \quad \forall \, v \in H^1_0(\Omega).
\end{equation}

The integral on the left above can be understood as a weighted form, where the weight depends on the solution $u$. Identity \eqref{eq:classical-MS} is the starting point for classical approaches to discretize the graph Plateau problem \cite{Ciarlet02, JoTh75, Rannacher}.

We now fix $s \in (0,1/2)$ and consider the $s$-perimeter operator $P_s$ given by \Cref{def:s-perimeter}. 
Like for the classical minimal surface problem, one may study the nonlocal minimal surface problem under the restriction of the domain being a cylinder. A difference between the problem we consider in this paper and its classical counterpart is that here imposition of Dirichlet data on the boundary of the domain becomes meaningless and thus we require that the exterior data can be written as a subgraph with respect to a fixed direction.
Concretely, from now on we consider $\Omega' = \Omega \times \mR$ with $\Omega \subset \mRd$  bounded. We assume that the exterior datum is the subgraph of some given function $g: \mRd \setminus \Omega \to \mR$,
\begin{equation} \label{E:Def-E0}
E_0 = \l\{ (x', x_{d+1}) \colon x_{d+1} < g(x'), \; x' \in \mRd \setminus \Omega \r\}.
\end{equation}

\begin{Remark}[assumptions on data] \label{R:assumptions}
Many of the results we describe in this paper are not optimal, in the sense that the assumptions can be weakened. In particular, this applies to the domain $\Omega$ and the Dirichlet datum $g$. About the latter, most of the theory can be carried out by assuming $g$ to be locally bounded and with some growth condition at infinity. However, in view of the proposed numerical method, we consider this exterior data function to be uniformly bounded and with bounded support.
More precisely, unless otherwise stated, from now on we assume that
\begin{equation} \label{E:assumptions} \begin{aligned}
& \Omega \mbox{ is a bounded Lipschitz domain;} \\
& g \in L^\infty(\mRd) \mbox{ with compact support.}
\end{aligned} \end{equation}
\end{Remark}

We leave all the technical discussion about the well-posedness of the nonlocal minimal graph problem to \Cref{appendix:perimeter}, but here we only mention two important features to take into account. The first one is that, in this setting, the notion of $s$-minimal set becomes meaningless, as every set $E$ that coincides with $E_0$ in $\Omega'$ satisfies $P_s(E,\Omega') =\infty$; the correct notion to consider is the one of {\em locally $s$-minimal set}. The second important feature is the existence of a locally $s$-minimal set in $\Omega'$ that coincides with the exterior datum \eqref{E:Def-E0}, and that actually corresponds to the subgraph of a function $u$ in $\Omega$, that is,
\begin{equation}\label{E:Def-E}
E \cap \Omega' = \l\{ (x', x_{d+1}) \colon x_{d+1} < u(x'), \; x' \in \Omega \r\}.
\end{equation}

\begin{Remark}[solving the graph nonlocal Plateau problem] \label{R:solutions}
\Cref{appendix:perimeter} explains that, in order to to find the locally $s$-minimal graph in $\Omega'$,
it suffices to take $M$ large enough, consider $\Omega_M = \Omega \times [-M,M]$, and then seek a function $u$ in the class
\begin{equation*}
\l\{ u \colon \mRd \to \mR : \  \Vert u \Vert_{L^{\infty}(\Omega)} \leq M, \ u = g \mbox{ in } \Omega^c \r\}
\end{equation*}
such that the set $E := \{ (x', x_{d+1}): x_{d+1} < u(x') \}$ satisfies 
\[
P_s(E, \Omega_M) \le P_s(F, \Omega_M)
\]
for every set $F$ that coincides with $E$ outside $\Omega_M$.
\end{Remark}

%%%%%%%%%%%%%%%%%%%%%%%%%%%%
\subsection{An energy functional}
%%%%%%%%%%%%%%%%%%%%%%%%%%%%
According to \eqref{E:Def-E}, the $s$-minimal sets we aim to approximate in this work are subgraphs. When restricting the fractional perimeter functional $P_s(\cdot, \Omega_M)$ to subgraphs of functions that coincide with some given $g$ in $\Omega^c$, it is convenient to work with an operator acting on the function $u$ rather than on the set $E$. 
More precisely,  consider the set $Q_{\Omega} = \l( \mRd \times \mRd \r) \setminus \l( {\Omega}^c \times {\Omega}^c \r)$ and the function $F_s \colon \mR \to \mR$, 
\begin{equation} \label{E:def_Fs}
F_s(\rho) := \int_0^\rho \frac{\rho-r}{\l( 1+r^2\r)^{(d+1+2s)/2}} dr.
\end{equation}

In order to study the Plateau problem for nonlocal minimal graphs, we introduce the energy
\begin{equation}\label{E:NMS-Energy-Graph}
I_s[u] = \iint_{Q_{\Omega}} F_s\l(\frac{u(x)-u(y)}{|x-y|}\r) \frac{1}{|x-y|^{d+2s-1}} \;dxdy.
\end{equation}
This functional is the nonlocal analogue of \eqref{E:MS-Energy-Graph}. Indeed, there exists a direct relation between the operator $I_s$ and the $s$-perimeter.
As pointed out in Remark \ref{R:solutions}, in order to find nonlocal minimal graphs on $\Omega$ it suffices to find minimizers of the fractional $s$-perimeter $P_s(E,\Omega_M)$ for $M$ sufficiently large.  The next proposition shows that, in the graph setting, if $M$ is large enough then the $s$-perimeter $P_s(E,\Omega_M)$ can be written as the sum of a term depending only on $u$ plus a term that is independent of $u$, albeit it blows up as $M \to \infty$. For completeness, we include a proof of this proposition in \Cref{appendix:energy}; we refer also to \cite[Proposition 4.2.8]{Lombardini-thesis}.

\begin{Proposition}[relation between $P_s$ and $I_s$] \label{prop:perimeter-energy}
Let $\Omega \subset \mRd$ be a bounded Lipschitz domain, $g \in L^\infty(\Omega^c)$, $M \ge \| g \|_{L^\infty(\Omega^c)}$, $\Omega' = \Omega \times \mR$ and $\Omega_M = \Omega \times [-M,M]$. Then, for every set $E$ of the type of \eqref{E:Def-E} with $\|u\|_{L^\infty(\Omega)} \le M$, it holds that
\[
P_s (E, \Omega_M) = I_s[u] + C(M,d,s,\Omega,g), 
\]
where $I_s[u]$ is given according to \eqref{E:NMS-Energy-Graph}. 
\end{Proposition}

An immediate consequence of this decomposition is that, if $M \ge \| g \|_{L^\infty(\Omega^c)}$, then the minimizer $u$ is independent of the truncation parameter $M$. Even though in the limit $M \to \infty$ the fractional perimeter is trivially equal to infinity, the function $u$ we compute has the `good credentials' to be regarded as a fractional minimal surface in the cylinder $\Omega'$. We recall that there cannot exist an $s$-minimal set on $\Omega'$ 
that coincides with the subgraph of a bounded function in $\Omega'^c$.

\begin{Remark}[growth of exterior data]
The functional $I_s$ may not be well-defined for functions that coincide with $g$ on $\Omega^c$ unless $g$ does not grow too fast at infinity. Nevertheless we point out that, as described in \eqref{E:assumptions}, in this work we assume that $g$ is bounded and with bounded support in $\mRd$. 
\end{Remark}

Our next goal is to define the correct space in which to look for minimizers of the energy functional $I_s$. We start with an auxiliary result.

\begin{Lemma}[energy bounds] \label{lemma:W2s}
Let $\Omega \subset \mRd$ be bounded.
Then, there exist constants $C_1(d,\Omega,s),$ $C_2(d,s)$ and $C_3(d,s)$ such that,
for every function $v \colon \mRd \to \mR$ it holds that
\begin{equation} \label{eq:norm_bound} \begin{aligned}
& |v|_{W^{2s}_1(\Omega)} \le C_1 + C_2 I_s[v], \\
& I_s[v] \le C_3 \iint_{Q_{\Omega}}  \frac{|v(x)-v(y)|}{|x-y|^{d+2s}} dxdy .
\end{aligned} \end{equation}
\end{Lemma}
\begin{proof}
From definition \eqref{E:def_Fs}, it follows immediately that $F_s(0) = 0$ and that
\[
F_s'(\rho) = \int_0^\rho \frac{1}{\l( 1+r^2\r)^{(d+1+2s)/2}} dr, \quad \forall \rho > 0.
\]
Thus, if we set the constant $C_3 = \int_0^\infty \frac{1}{\l( 1+r^2\r)^{(d+1+2s)/2}} dr$, we deduce that
\[
F_s(\rho) \le C_3 \rho \quad \forall \rho \ge 0.
\]
This implies the second inequality in \eqref{eq:norm_bound}.
	
On the other hand, the first estimate in \eqref{eq:norm_bound}, with constant $C_1 = \iint_{\Omega\times\Omega} \frac{dxdy}{|x-y|^{d+2s-1}} < \infty$ because $\Omega \subset \mRd$ is bounded, is a consequence of the bound 
\begin{equation} \label{eq:lower-bound-Fs}
\rho \le 1 + C_2 F_s(\rho) \quad \forall \rho \ge 0.
\end{equation}
It is obvious that such a bound holds for $0 \le \rho \le 1$, whereas if $\rho > 1$, we have $F'_s (\rho)  \ge F'_s (1)$ and therefore, $F_s(\rho) > F'_s(1) (\rho - 1)$. The desired inequality follows with constant $C_2 = 1/ F'_s(1)$.
\end{proof}

Taking into account the lemma we have just proved, we introduce the natural spaces in which to look for nonlocal minimal graphs.

\begin{Definition}[space $\mathbb{V}^g$] \label{Def:space-Vg}  
Given $g \colon \Omega^c \to \mR$, we consider
\begin{equation*}
\mathbb{V}^g = \{ v \colon \mRd \to \mR \; \colon \; v\big|_\Omega \in W^{2s}_1(\Omega), \ v = g \text{ in } {\Omega}^c, \ |v|_{\mathbb{V}^g} < \infty \}, 
\end{equation*}
equipped with the norm
\[
\| v \|_{\mathbb{V}^g } = \| v \|_{L^1(\Omega)} + | v |_{\mathbb{V}^g },
\]
where
\[
| v |_{\mathbb{V}^g} := \iint_{Q_{\Omega}}  \frac{|v(x)-v(y)|}{|x-y|^{d+2s}} dxdy.  
\]
\end{Definition}

In the specific case where $g$ is the zero function, we denote the resulting space $\mathbb{V}^g$ by $\mathbb{V}^0$. The set $\mathbb{V}^g$ can also be understood as that of functions in $W^{2s}_1 (\Omega)$ with `boundary value' $g$. Indeed, we point out that in Definition \ref{Def:space-Vg} we do not require $g$ to be a function in $W^{2s}_1(\Omega^c)$ (in particular, $g$ may not decay at infinity). The seminorm $|\cdot |_{\mathbb{V}^g}$ does not take into account interactions over $\Omega^c \times \Omega^c$, because these are fixed for the applications we consider.

As stated in the next Proposition, given a Dirichlet datum $g$, the space $\mathbb{V}^g$ is the natural domain of the energy $I_s$.

\begin{Proposition}[energy domain] \label{Prop:domain-energy} Let $s \in (0,1/2)$ and $\Omega, g$ be given according to \eqref{E:assumptions}. Let $v \colon \mRd \to \mR$ be such that $v = g$ in $\Omega^c$. Then, $v \in \mathbb{V}^g $ if and only if $I_s[v] < \infty$.
\end{Proposition}
\begin{proof}
The claim follows easily from Lemma \ref{lemma:W2s}. Let $v$ be a function that coincides with $g$ in $\Omega^c$. Then, if $v \in \mathbb{V}^g$, the second estimate in \eqref{eq:norm_bound} gives $I_s[v] < \infty$, because $|v|_{\mathbb{V}^g} < \infty$. 

Reciprocally, if $I_s[v] < \infty$, the first inequality in \eqref{eq:norm_bound} implies that $| v \big|_\Omega|_{W^{2s}_1(\Omega)} < \infty$. Fix $R > 0$ such that $\Omega \subset B_{R/2}$; because of \eqref{eq:lower-bound-Fs} and the Lipschitz continuity of $F_s$, integrating over $\Omega \times (B_R \setminus \Omega)$, we obtain
\begin{equation} \label{eq:L1-bound}
\begin{aligned}
\frac{\| v \|_{L^1(\Omega)}}{R^{2s}} & \lesssim \iint_{\Omega \times (B_R \setminus \Omega)} \l(1 + F_s \l( \frac{|v(x)|}{|x-y|} \r)\r) \frac{dx dy}{|x-y|^{d+2s-1}} \\
& \lesssim  I_s[v] + \iint_{\Omega \times (B_R \setminus \Omega)}  \frac{1+|g(y)|}{|x-y|^{d+2s}} dx dy. 
\end{aligned}
\end{equation}
The last integral in the right hand side above is finite because  $\| g \|_{L^\infty} < \infty$. Therefore, $v \in L^1(\Omega)$. To deduce that $|v|_{\mathbb{V}^g} < \infty$, we split the integral, use the triangle inequality, integrate in polar coordinates and apply Hardy's inequality \cite[Theorem 1.4.4.4]{Grisvard} to derive
\begin{equation*}
\begin{aligned}
|v|_{\mathbb{V}^g} & \le | v \big|_\Omega|_{W^{2s}_1(\Omega)} + 2 \iint_{\Omega\times\Omega^c} \frac{|v(x)|}{|x-y|^{d+2s}} dx dy + 2 \iint_{\Omega\times\Omega^c} \frac{|g(y)|}{|x-y|^{d+2s}} dx dy \\
& \lesssim | v \big|_\Omega|_{W^{2s}_1(\Omega)} + \int_{\Omega} \frac{|v(x)|}{\dist(x,\pO)^{2s}} dx + \| g \|_{L^\infty(\Omega^c)} \lesssim \| v \big|_\Omega\|_{W^{2s}_1(\Omega)} + \| g \|_{L^\infty(\Omega^c)}.
\end{aligned} 
\end{equation*}
This proves that $v \in \mathbb{V}^g$ and concludes the proof.
\end{proof}

Taking into account \Cref{prop:perimeter-energy} and \Cref{Prop:domain-energy}, we obtain a characterization of nonlocal minimal graphs (see also \cite[Theorem 4.1.11]{Lombardini-thesis}).

\begin{Corollary}[relation between minimization problems] \label{Cor:characterization-NMS}
Let $s \in (0,1/2)$ and $\Omega, g$ satisfy \eqref{E:assumptions}. Given a function $u \colon \mRd \to \mR$ that coincides with $g$ in $\Omega^c$, consider the set $E$ given by \eqref{E:Def-E}. Then, $E$ is locally $s$-minimal in $\Omega' = \Omega \times \mR$ if and only if $u$ minimizes the energy $I_s$ in the space $\mathbb{V}^g$.
\end{Corollary}

The functional $I_s$ is strictly convex, because the weight $F_s$ appearing in its definition (cf. \eqref{E:def_Fs}) is strictly convex as well. Therefore, we straightforwardly deduce the next result.
\begin{Corollary}[uniqueness]
Under the same hypothesis as in Corollary \ref{Cor:characterization-NMS}, there exists a unique locally $s$-minimal set.
\end{Corollary}

We conclude this section with a result about the regularity of the minimizers of $I_s$.
In spite of being prone to be discontinuous across the boundary, minimal graphs are smooth in the interior of the domain. The following theorem is stated in \cite[Theorem 1.1]{CabreCozzi2017gradient}, where an estimate for the gradient of the minimal function is derived. Once such an estimate is obtained, the claim follows by the arguments from \cite{Barrios2014bootstrap} and \cite{Figalli2017regularity}.

\begin{Theorem}[interior smoothness of nonlocal minimal graphs] \label{thm:smoothness}
Assume $E \subset \mathbb{R}^{d+1}$ is an $s$-minimal surface in $\Omega' = \Omega \times \mathbb{R}$, given by the subgraph of a measurable function $u$ that is bounded in an open set $\Lambda \supset \Omega$.
Then, $u \in C^\infty (\Omega)$.
\end{Theorem}

%%%%%%%%%%%%%%%%%%%%%%%%%%%%
\subsection{Weak formulation}
%%%%%%%%%%%%%%%%%%%%%%%%%%%%
In order to define the proper variational setting to study the nonlocal minimal graph problem, we introduce the function $G_s \colon \mR \to \mR$,
\begin{equation} \label{E:DEF-Gs}
G_s(\rho) = \int_0^\rho (1+r^2)^{-(d+1+2s)/2} dr = F'_s(\rho).
\end{equation}
We recall that $s \in (0,1/2)$. Clearly, $G_s$ is an odd and uniformly bounded function:
\begin{equation} \label{E:bound-Gs}
| G_s (\rho) | \le K := \int_{0}^\infty (1+r^2)^{-(d+1+2s)/2} dr = \frac{\Gamma\l( \frac{d+2s}{2}\r) \sqrt{\pi}}{2 \Gamma \l(\frac{d+1+2s}{2}\r)}.
\end{equation}
The constant $K$ has already appeared in the proof of \Cref{lemma:W2s}, under the label $C_3(d,s)$. The last equality above follows from the substitution $t = (1+r^2)^{-1}$ and the basic relation between the beta and gamma functions, $B(x,y) =\frac{\Gamma(x)\Gamma(y)}{\Gamma(x+y)}$. 

Given a function $u \in \mathbb{V}^g$, we take the bilinear form $a_u \colon \mathbb{V}^g \times \mathbb{V}^0 \to \mathbb{R}$,
\begin{equation} \label{E:def-a}
a_u(w,v) := \iint_{Q_{\Omega}} \Gts\l(\frac{u(x)-u(y)}{|x-y|}\r) \frac{(w(x)-w(y))(v(x)-v(y))}{|x-y|^{d+1+2s}}dx dy, 
\end{equation}
where $\Gts(\rho) = \int_0^1 (1+ \rho^2 r^2)^{-(d+1+2s)/2} dr$ and hence it satisfies $\rho \Gts(\rho) = G_s(\rho)$. 

To obtain a weak formulation of our problem, we compute the first variation of \eqref{E:NMS-Energy-Graph}, that yields
\[
\delta I_s [u] (v) = a_u(u,v) \mbox{ for all } \quad v \in \mathbb{V}^0.
\]
Thus, we seek a function $u \in \mathbb{V}^g$ such that  
\begin{equation}\label{E:WeakForm-NMS-Graph}
a_u(u,v) = 0 \quad \mbox{ for all }  v \in \mathbb{V}^0.
\end{equation}

Another approach --at least formal-- to derive problem \eqref{E:WeakForm-NMS-Graph} is to write it as the weak form of a suitable Euler-Lagrange equation. More precisely, assuming that the set $E$ is the subgraph of a function $u$, this can be written as the following nonlocal and nonlinear equation \cite{Barrios2014bootstrap}
\begin{equation}\label{E:EL-eq-Graph}
H_s[E](x) = \PV \int_{\mRd} G_s\l(\frac{u(x)-u(y)}{|x-y|}\r) \frac{1}{|x-y|^{d+2s}} \;dy = 0,
\end{equation}
in a viscosity sense, for every $x \in \Omega$.
With some abuse of notation, we let $H_s[u]$ represent $H_s[E]$ when $E$ is the subgraph of $u$. Therefore, $u$ solves the Dirichlet problem
\begin{equation}\label{E:PDE-NMS}
\l\{\begin{array}{rl}
H_s[u](x) = 0 \ \quad & \quad x \in \Omega, \\
u(x)  = g(x) & \quad x \in \mRd \setminus \Omega.
\end{array}\r.
\end{equation}
In this regard, the weak formulation of \eqref{E:PDE-NMS} is set by multiplying it by a test function, integrating and taking advantage of the fact that $G_s$ is an odd function. This corresponds to \eqref{E:WeakForm-NMS-Graph}. 

We finally point out that \eqref{E:WeakForm-NMS-Graph} can be interpreted as a fractional diffusion problem of order $s+1/2$ with weights depending on the solution $u$ and fixed nonhomogeneous boundary data; this is in agreement with the local case \eqref{eq:classical-MS}. Like for the classical minimal graph problem, the nonlinearity degenerates wherever the Lipschitz modulus of continuity of $u$ blows up. We expect this to be the case as $\dist(x,\partial\Omega) \to 0$, as this has been shown to be the generic behavior in one-dimensional problems \cite{DipiSavinVald19}.

%%%%%%%%%%%%%%%%%%%%%%%%%%%%
%%%%%%%%%%%%%%%%%%%%%%%%%%%%
\section{Numerical Method} \label{sec:numerical-method}
%%%%%%%%%%%%%%%%%%%%%%%%%%%%
%%%%%%%%%%%%%%%%%%%%%%%%%%%%

This section introduces the framework for the discrete nonlocal minimal graph problems under consideration. We set the notation regarding discrete spaces and analyze their approximation properties by resorting to a quasi-interpolation operator of Cl\'ement type. We include a brief discussion on the solution of the resulting nonlinear discrete problems.

\subsection{Finite element discretization} \label{ss:FE_discretization}
As discussed in \Cref{R:assumptions}, in this work we assume that $g$ is a function with bounded support. Concretely, we assume that 
\begin{equation} \label{E:support}
\mbox{supp}(g) \subset \Lambda \mbox{ for some bounded set } \Lambda.
\end{equation}
Approximations for unboundedly supported data are discussed in a forthcoming paper by the authors \cite{BoLiNo19computational}. Without loss of generality, we may assume that $\Lambda = B_R(0)$ is a ball of radius $R$ centered at the origin.

We consider a family $\{\Th \}_{h>0}$ of conforming and simplicial meshes of $\Lambda$, that we additionally require to exactly mesh $\Omega$.
Moreover, we assume this family to be shape-regular, namely:
\[
  \sigma = \sup_{h>0} \max_{T \in \Th} \frac{h_T}{\rho_T} <\infty,
\]
where $h_T = \mbox{diam}(T)$ and $\rho_T $ is the diameter of the largest ball contained in $T$. As usual, the subindex $h$ denotes the mesh size, $h = \max_{T \in \Th} h_T$.
The set of vertices of $\Th$ will be denoted by $\mathcal{N}_h$, and $\varphi_i$ will denote the standard piecewise linear Lagrangian basis function associated to the node $\x_i \in \mathcal{N}_h$. In this work we assume that the elements are closed sets. Thus, the star or first ring of an element $T \in \Th$ is given by
\[
  S^1_T = \bigcup \left\{ T' \in \Th \colon \oT \cap \oTp \neq \emptyset \right\}.
\]
We also introduce the star or second ring of $S^1_T$,
\[
  S^2_T = \bigcup \left\{ T' \in \Th \colon S^1_T \cap \oTp \neq \emptyset \right\},
\]
and the star of the node $\x_i \in \mathcal{N}_h$, $S_i = \mbox{supp}(\varphi_i)$.
We split the mesh nodes into two disjoint sets, consisting of either vertices in $\Omega$ and in $\Omega^c$,
\[
\mathcal{N}_h^\circ = \left\{ \x_i \colon \x_i \in \Omega \right\}, \qquad  \mathcal{N}_h^c = \left\{ \x_i \colon \x_i \in \Omega^c \right\}.
\]
We emphasize that, because $\Omega$ is an open set, nodes on $\pO$ belong to $\mathcal{N}_h^c$.

The discrete spaces we consider consist of continuous piecewise linear functions in $\Lambda$. Indeed, we set
\[
\mathbb{V}_h =  \{ v \in C(\Lambda) \colon v|_T \in \mathcal{P}_1 \; \forall T \in \Th \}.
\]
For this work, we make use of certain Cl\'ement-type interpolation operators on $\mathbb{V}_h$. To account for boundary data, given an integrable function $g \colon \Lambda \setminus \Omega \to \mR$, we define 
\[
\mathbb{V}_h^g = \{ v \in \mathbb{V}_h \colon \ v|_{\Lambda \setminus \Omega} = \Cl^c g\}.
\]
Here, $\Cl^c$ denotes the {\em exterior} Cl\'ement interpolation operator in $\Omega^c$, defined as
\[
\Cl^c g := \sum_{\x_i \in \Nhc} (\Cl^{\x_i} g) (\x_i) \; \varphi_i,
\]
where $\Cl^{\x_i} g$ is the $L^2$-projection of  $g\big|_{S_i \cap \Omega^c}$ onto $\mathcal{P}_1(S_i \cap \Omega^c)$. Thus, $\Cl^c g (\x_i)$ coincides with the standard Cl\'ement interpolation of $g$ on $\x_i$ for all nodes $\x_i$ such that $S_i \subset \mRd \setminus \overline{\Omega}$. On the other hand, for nodes on the boundary of $\Omega$, $\Cl^c$ only averages over the elements in $S_i$ that lie in $\Omega^c$. Although $\Cl^c g$ only takes into account values of $g$ in $\Omega^c$, the support of $\Cl^c g$ is not contained in $\Omega^c$, because $\varphi_i$ attains nonzero values in $\Omega$ for $x_i \in \pO$.

Using the same convention as before, in case $g$ is the zero function, we write the corresponding space as $\mathbb{V}_h^0$. Also, we define the {\em interior} Cl\'ement interpolation operator $\Cl^\circ \colon L^1(\Omega) \to \mathbb{V}^0_h$,
\[
\Cl^\circ v := \sum_{\x_i \in \mathcal{N}_h^\circ} (\Cl^{\x_i} v) (\x_i) \; \varphi_i,
\]
where $\Cl^{\x_i} v$ is the $L^2$-projection of  $v \big|_{\Omega}$ onto $\mathcal{P}_1(S_i)$.

\begin{Remark}[discrete functions are continuous]
Even though nonlocal minimal surfaces can develop discontinuities across $\partial\Omega$
--recall the sticky behavior commented in \Cref{R:stickiness}--  the discrete spaces we consider consist of continuous functions. This does not preclude the convergence of the numerical scheme we propose in `trace blind' fractional Sobolev spaces. 
Furthermore, the strong imposition of the Dirichlet data simplifies both the method and its analysis. The use of discrete spaces that capture discontinuities across the boundary of the domain is subject of ongoing work by the authors.
\end{Remark}

With the notation introduced above, the discrete counterpart of \eqref{E:WeakForm-NMS-Graph} reads: find $u_h \in \mathbb{V}^g_h$ such that
\begin{equation}\label{E:WeakForm-discrete}
a_{u_h}(u_h, v_h) = 0 \quad \mbox{for all } v_h \in \mathbb{V}^0_h. 
\end{equation}

Due to our assumption \eqref{E:assumptions} on the datum $g$, it follows immediately that $u_h$ is a solution of \eqref{E:WeakForm-discrete} if and only if $u_h$ minimizes the  strictly convex energy $I_s[u_h]$ over the discrete space $\mathbb{V}_h^g$. This leads to the existence and uniqueness of solutions to the discrete problem \eqref{E:WeakForm-discrete}.

\subsection{Localization} An obvious difficulty when trying to prove interpolation estimates in 
fractional Sobolev spaces is that their seminorms are non-additive with respect to disjoint domain partitions. Here we state a localization result, proved by Faermann \cite{Faermann2,Faermann} in the case $p=2$. For brevity, since the proof for $p \neq 2$ follows by the same arguments as in those references, we omit it.

\begin{Proposition}[localization of fractional-order seminorms]
\label{prop:faermann}
Let $s \in (0,1)$, $p \in [1,\infty)$, and $\Omega$ be a bounded Lipschitz domain. Let $\Th$ denote a mesh as above. Then, for any $v \in W^s_p(\Omega)$ there holds
\begin{equation}
\label{E:faermann}
  |v|_{W^s_p(\Omega)} \leq \left[\sum_{T \colon T \subset \overline\Omega} \iint_{T \times (S^1_T \cap \Omega)} \frac{|v (x) - v (y)|^p}{|x-y|^{d+sp}} dy  dx 
    + C(\sigma) \, \frac{2^p \w_{d-1}}{s p h_T^{s p}} \, \| v \|^p_{L^p(T)} \right]^{1/p}.
\end{equation}
Above, $\w_{d-1}$ denotes the measure of the $(d-1)$-dimensional unit sphere.
\end{Proposition}

This localization of fractional-order seminorms is instrumental for our error analysis. It implies that, in order to prove approximability estimates in $W^s_p(\Omega)$, it suffices to produce {\em local} estimates in patches of the form $T\times S^1_T$ and scaled local $L^p(T)$ estimates  for every $T \in \Th$.

%%%%%%%%%%%%%%%%%%%%%%%%%%%%
\subsection{Interpolation operator}
%%%%%%%%%%%%%%%%%%%%%%%%%%%%
Here we define a quasi-interpolation operator that plays an important role in the analysis of the discrete scheme proposed in this paper. Such an operator combines the two Cl\'ement-type interpolation operators introduced in the previous subsection.  More precisely, we set $\interp \colon L^1(\mRd) \to \mathbb{V}_h^g$,
\begin{equation} \label{E:interpolation}
\interp v = \Cl^\circ \l(v \big|_\Omega \r)  + \Cl^c g .
\end{equation}

Using standard arguments for Cl\'ement interpolation, we obtain local approximation estimates in the interior of $\Omega$.

\begin{Proposition}[local interpolation error]
\label{prop:local_interpolation_estimate}
Let $s \in (0,1)$, $p \ge 1$, $s < t \le 2$. Then, for all $T \in \Th$ it holds
\[
\| v - \interp v \|_{L^p(T)} \lesssim h^{t} |v|_{W^t_p(S^1_T)},
\]
and
\[
\l(\iint_{T \times S_T^1} \frac{|(v - \interp v (x)) - (v - \interp v (y))|^p}{|x - y|^{d + sp}} dy dx \r)^\frac1p \lesssim h^{t-s} |v|_{W^t_p(S^2_T)} . 
\]
\end{Proposition}

From \Cref{Cor:characterization-NMS} and \Cref{thm:smoothness} we know that, under suitable assumptions, minimal graphs are $W^{2s}_1$-functions that are locally smooth in $\Omega$. These conditions are sufficient to prove the convergence of the interpolation operator $\interp$.

\begin{Proposition}[interpolation error] \label{prop:interpolation_estimate}
	Let $s \in (0,1)$, and $p \ge 1$ be such that $sp < 1$. Assume that $\Omega$ and $g$ satisfy \eqref{E:assumptions} and \eqref{E:support}. Then, for all 
	$v \colon \mRd \to \mR$	satisfying $v \big|_\Omega \in W^s_p(\Omega)$ and $v = g$ in $\Omega^c$, 
\[
\iint_{Q_{\Omega}}  \frac{\l| (\interp v -v)(x) - (\interp v -v)(y) \r|^p}{|x-y|^{d+sp}} \; dxdy \to 0 \ \textrm{ as } h \to 0.
\]
\end{Proposition}
\begin{proof}
In first place, we split
\[ \begin{aligned}
	\iint_{\Omega\times\Omega^c}  \frac{\l| (\interp v -v)(x) - (\interp v -v)(y) \r|^p}{|x-y|^{d+sp}} \; dxdy \lesssim &  \iint_{\Omega\times\Omega^c}  \frac{\l| \interp v (x) -v(x) \r|^p}{|x-y|^{d+sp}} \; dxdy  \\ 
	& + \iint_{\Omega\times\Omega^c}  \frac{\l| \Cl^c g (y) -g (y) \r|^p}{|x-y|^{d+sp}} \; dxdy .
\end{aligned} 
\]
	
Given $x \in \Omega$, we have $\int_{\Omega^c} \frac{1}{|x-y|^{d+sp}} dy \lesssim d(x,\pO)^{-sp},$ and since $\interp v -v \in W^s_p(\Omega)$, we invoke the Hardy inequality \cite[Theorem 1.4.4.4]{Grisvard} to deduce that
\[
\iint_{\Omega\times\Omega^c}  \frac{\l| (\interp v -v)(x)\r|^p}{|x-y|^{d+sp}} \; dxdy \lesssim \int_{\Omega}  \frac{\l| (\interp v -v)(x)\r|^p}{d(x,\pO)^{sp}} \; dx \lesssim \| \interp v -v \|_{W^s_p(\Omega)}^p.
\]
Since $g$ is uniformly bounded, we first claim that $\Cl^c g \to g$ a.e. in $\Omega^c$ as $h \to 0$. Indeed, for every $y \in T \subset \Omega^c$, we express $\Cl^c g(y)$ as a linear combination of $\Cl^c g(\x_i)$, where $\{ \x_i \}$ are the vertices of $T$, and deduce that
\[
\Cl^c g(y) = \int_{S_T^1} \vp^*_{y}(x) g(x) dx, 
\]
for some function $\vp^*_{y}$ satisfying $\int_{S_T^1} \vp^*_{y}(x) dx = 1$ and$\Vert \vp^*_{y} \Vert_{L^{\infty}(S_T^1)} \le C(d,\sigma) h^{-d}$. Since $S_T^1 \subset \overline{B}_{2h}(y)$, for every Lebesgue point y of $g$ we have
\[\begin{aligned}	
\big| (\Cl^c g)(y) - g(y) \big| &= \l| \int_{S_T^1} \vp^*_{y}(x) \big( g(x) - g(y) \big) dx\r| \le \Vert \vp^*_{y} \Vert_{L^{\infty}(S_T^1)} \int_{B_{2h}(y)} |g(x) -g(y)|dx \\
&\lesssim \frac{1}{\l| B_{2h}(y) \r| }\int_{B_{2h}(y)} |g(x) -g(y)|dx \to 0 \textrm{ as } h \to 0.
\end{aligned} 
\]
By the Lebesgue differentiation theorem, almost every $y \in \Omega^c$ is a Lebesgue point of $g$, and therefore 
	\begin{equation}\label{E:OmegaC-ae-converge}
	\Cl^c g \to g \; \textrm{ a.e. in } \Omega^c \textrm{ as } h \to 0.
	\end{equation}
	In addition, it follows from $\Vert \vp^*_{y} \Vert_{L^{\infty}(S_T^1)} \le C(d,\sigma) h^{-d}$ that
	$\Vert \Cl^c g \Vert_{L^{\infty}(\mRd)} \lesssim \Vert g \Vert_{L^{\infty}(\Omega^c)}$, and hence
\[
\iint_{\Omega\times\Omega^c} \frac{\l| \Cl^c g (y) -g (y) \r|^p}{|x-y|^{d+sp}}  dxdy \lesssim \iint_{\Omega\times\Omega^c}  \frac{\Vert g \Vert^p_{L^{\infty}(\Omega^c)}}{|x-y|^{d+sp}} < \infty.
\]
Applying the Lebesgue Dominated convergence theorem, we obtain
\[
\iint_{\Omega\times\Omega^c} \frac{\l| \Cl^c g (y) -g (y) \r|^p}{|x-y|^{d+sp}}  dxdy \to 0 \textrm{ as } h \to 0.
\]
Therefore, we have shown that
\[
\iint_{\Omega\times\Omega^c}  \frac{\l| (\interp v -v)(x) - (\interp v -v)(y) \r|^p}{|x-y|^{d+sp}} \; dxdy \lesssim  \| \interp v -v \|_{W^s_p(\Omega)}^p + o(1),
\]
and thus we just need to bound the interpolation error in $W^s_p(\Omega)$.

We write $\interp v -v = \Big( \Cl^\circ \l(v \big|_\Omega \r) - v \Big) + \Cl^c g$, and split
\[
\big\|\interp v -v \big\|_{W^s_p(\Omega)}^p \lesssim
\big\| \Cl^\circ \l(v \big|_\Omega \r) - v \big\|_{W^s_p(\Omega)}^p + 
\big\| \Cl^c g \big\|_{W^s_p(\Omega)}^p.
\]	
Using the localization estimate \eqref{E:faermann}, we bound $\big\| \Cl^c g \big\|_{W^s_p(\Omega)}^p$ by
\[ \begin{aligned}
\big\| \Cl^c g \big\|_{W^s_p(\Omega)}^p \leq
\sum_{T \colon T \subset \overline\Omega} \bigg[ & \iint_{T \times (S^1_T \cap \Omega)} \frac{| \Cl^c g (x) - \Cl^c g (y)|^p}{|x-y|^{d+sp}} dy  dx \\ &  + \l(1 + C(\sigma) \, \frac{2^p \w_{d-1}}{s p h_T^{s p}} \r) \, \| \Cl^c g \|^p_{L^p(T)} \bigg]
\end{aligned} \]
Recalling $\Vert \Cl^c g \Vert_{L^{\infty}(\mRd)} \lesssim \Vert g \Vert_{L^{\infty}(\Omega^c)}$ and using an inverse inequality, we have
\[ \begin{aligned}
\big\| \Cl^c g \big\|_{W^s_p(\Omega)}^p \lesssim
\sum_{T \subset \overline \Omega \colon S^1_T \cap \Omega^c \neq \emptyset} \bigg[  h_T^{-sp+d} \Vert \Cl^c g \Vert_{L^{\infty}(\mRd)}^p	\bigg] \lesssim \Vert g \Vert_{L^{\infty}(\Omega^c)} \!\!\!\!\!
\sum_{T \subset \overline \Omega \colon S^1_T \cap \Omega^c \neq \emptyset} 
h_T^{-sp+d},
\end{aligned} \]
for $h$ small enough. The sum in the right hand side above can be straightforwardly estimated by 
\[ \begin{aligned}
\sum_{T \subset \overline \Omega \colon S^1_T \cap \Omega^c \neq \emptyset} 
h_T^{-sp+d}	& \lesssim \sum_{T \subset \overline{\Omega}: \; T \cap \Omega^c \neq \emptyset} h_T^{-sp+d} \le h^{1-sp} \sum_{T \subset \overline{\Omega}: \; T \cap \Omega^c \neq \emptyset} h_T^{d-1} \\ 
& \lesssim h^{1-sp}\; \mathcal{H}^{d-1}(\pO),
\end{aligned} \]
where $\mathcal{H}^{d-1}$ is the $(d-1)$-dimensional Hausdorff measure. This establishes that $\big\| \Cl^c g \big\|_{W^s_p(\Omega)}^p \to 0$ as $h \to 0$.		
	
It only remains to show that $\big\| \Cl^\circ \l(v \big|_\Omega \r) - v \big\|_{W^s_p(\Omega)}^p \to 0$ as $h \to 0$. For simplicity, we write $\Cl^\circ v$ instead of $\Cl^\circ \l(v \big|_\Omega \r)$. Since $\Cl^\circ v$ is a continuous linear operator from $W^s_p(\Omega)$ to $W^s_p(\Omega)$ with
\[
\| \Cl^\circ v \|_{W^s_p(\Omega)} \le C(d,s,p,\sigma,\Omega) \| v \|_{W^s_p(\Omega)},
\]
it suffices to prove the convergence for $v \in C^{\infty}(\Oc)$. We use the localization estimate \eqref{E:faermann} for $\big\| \Cl^\circ v - v \big\|_{W^s_p(\Omega)}^p$ and write
\[ \begin{aligned}
\big\| \Cl^\circ v - v \big\|_{W^s_p(\Omega)}^p & \leq  
\sum_{T \colon T \subset \overline\Omega} \bigg[ \iint_{T \times (S^1_T \cap \Omega)} \frac{| (\Cl^\circ v - v) (x) - (\Cl^\circ v - v) (y)|^p}{|x-y|^{d+sp}} dy  dx \\
& \qquad \qquad \qquad  + \l(1 + C(\sigma) \, \frac{2^p \w_{d-1}}{s p h_T^{s p}} \r) \, \| \Cl^\circ v - v \|^p_{L^p(T)} \bigg] \\
&  =: \sum_{T  \colon S_T^2 \subset \Omega} I_T(v) \ + \sum_{T \subset \overline\Omega \colon S_T^2  \cap \Omega^c \neq \emptyset} I_T(v).
\end{aligned} \]
	
On the one hand, we point out that, because
\[
\l| \bigcup_{T \subset \overline\Omega \colon S^2_T \cap \Omega^c \neq \emptyset} T \r| \approx h,
\]
and $v \in W^s_p(\Omega)$, we have
\[
\sum_{T \subset \Omega \colon S_T^2  \cap \widetilde{\Omega}^c \neq \emptyset} I_T(v) \le C(d,s,p,\sigma) \sum_{T \subset \Omega \colon S_T^2  \cap \widetilde{\Omega}^c \neq \emptyset} \|v\|_{W^s_p(S^2_T \cap \Omega)}^p \to 0, \mbox{ as } h \to 0.
\]

On the other hand, over the elements $T$ such that $S_T^2 \subset \Omega$, \Cref{prop:local_interpolation_estimate} gives
\[
\sum_{T \colon S_T^2 \subset \widetilde{\Omega}} I_T(v) \lesssim h^{p(t-s)} \sum_{T \colon S_T^2 \subset \widetilde{\Omega}} |v|_{W^t_p(S^2_T)}^p \lesssim h^{p(t-s)} |v|^p_{W^t_p(\Omega)},
\]
for every $v$ smooth in $\Omega$. This finishes the proof.
\end{proof}

If, under the same conditions as in \Cref{prop:interpolation_estimate}, we add the hypothesis that $v$ is smoother than $W^s_p(\Omega)$, then it is possible to derive interpolation rates.

\begin{Proposition}[interpolation rate] \label{prop:interpolation_estimate-2}
Assume that $v \in W^t_p(\Lambda)$ for some $t \in (s,2]$. Then, 
\[
\iint_{Q_{\Omega}}  \frac{\l| (\interp v -v)(x) - (\interp v -v)(y) \r|^p}{|x-y|^{d+sp}} \; dxdy \lesssim h^{p(t-s)} |v|^p_{W^t_p(\Lambda)} .
\]
\end{Proposition}
\begin{proof}
We split $Q_{\Omega}$ into the sets
\[
Q_{\Omega} \subset \big( \Lambda \times \Lambda \big)
\cup \Big( \big( \Omega^c \setminus \Lambda \big) \times \Omega \Big) \; \cup \; \Big( \Omega \times \big( \Omega^c \setminus \Lambda \big) \Big).
\]
The estimate in $\Lambda \times \Lambda$ is standard and follows along the lines of the estimates for $\Vert \interp v - v \Vert_{W^s_p(\Omega)}$ in \Cref{prop:interpolation_estimate}. The estimate in $\big( \Omega^c \setminus \Lambda \big) \times \Omega$ reduces to $\Vert \interp v - v \Vert_{L^p(\Omega)}$ because $\dist(\Omega^c \setminus \Lambda, \Omega) > 0$ and $v$ is zero in $\Omega^c \setminus \Lambda$.
\end{proof}

%%%%%%%%%%%%%%%%%%%%%%%%%%%%
\subsection{Numerical schemes} \label{SS:schemes}
%%%%%%%%%%%%%%%%%%%%%%%%%%%%
We briefly include some details about the implementation and solution of the discrete problem \eqref{E:WeakForm-discrete}. In first place we point out that we can compute $a_{u_h}(u_h, v_h)$ for any given $u_h \in \mathbb{V}_h^g, \ v_h \in \mathbb{V}_h^0$ by following the implementation techniques from \cite{AcosBersBort2017short, AcosBort2017fractional}.
Further details on the quadrature rules employed and the treatment of the discrete form $a_{u_h}(u_h, v_h)$ can be found in \cite{BoLiNo19computational}.

In order to solve the nonlinear discrete problem we resort to two different approaches: a semi-implicit $H^\alpha$-gradient flow and a damped Newton algorithm. For the former, we consider $\alpha \in [0,1)$ (with the convention that $H^0 (\Omega) = L^2(\Omega)$), fix a step size $\tau > 0$ and, given an initial guess $u_h^0$, we solve the following equation in every step,
\begin{equation}\label{E:semi-implicit-GF}
\frac1{\tau} \langle {u^{k+1}_h - u^k_h} \ , \ v_h \rangle_{H^{\alpha}(\Omega)} = -a_{u^k_h}(u^{k+1}_h \ , \ v_h), \qquad \forall v_h \in \mathbb{V}^0_h .
\end{equation}

For the damped Newton method, we take the first variation of $a_{u_h}(u_h,v_h)$, $\frac{\delta a_{u_h}(u_h,v_h)}{\delta u_h}(w_h)$, which is well-defined for all $u_h \in \mathbb{V}^g_h$, $v_h, w_h \in \mathbb{V}^0_h$. We point out that the analogue of this variation at the continuous level is not well-defined. The resulting step is obtained by solving for $w_h^k$ the equation
\begin{equation}\label{E:damped-Newton-wh}
\frac{\delta a_{u_h}(u_h^k, v_h)}{\delta u_h^k}(w_h^k)
= -a_{u_h^k}(u_h^k, v_h), \qquad \forall v_h \in \mathbb{V}^0_h
\end{equation}
and performing a line search to determine the step size. We refer the reader to \cite{BoLiNo19computational} for full details on these algorithms and further discussion on their performance.

%%%%%%%%%%%%%%%%%%%%%%%%%%%%
%%%%%%%%%%%%%%%%%%%%%%%%%%%%
\section{Convergence} \label{sec:convergence}
%%%%%%%%%%%%%%%%%%%%%%%%%%%%
%%%%%%%%%%%%%%%%%%%%%%%%%%%%

In this section, we prove the convergence of the discrete solution $u_h$ without assumptions on the regularity of the nonlocal minimal graphs.

We first prove that the discretization proposed in \Cref{ss:FE_discretization} is energy-consistent. Due to our assumption that $g \in C_c(\Omega^c)$, we know from Proposition \ref{Prop:domain-energy} that the energy minimizing function $u \in W^{2s}_1(\Omega)$, while Theorem  \ref{thm:smoothness} guarantees that for any region $\widetilde{\Omega} \Subset \Omega$, $u \in W^{2t}_1(\widetilde{\Omega})$ for every $t \in \mR$.
These two properties are sufficient to guarantee the consistency of the discrete energy.

\begin{Theorem}[energy consistency] \label{thm:consistency}
Let $s \in (0,1/2)$, and assume that $\Omega$ and $g$ satisfy \eqref{E:assumptions} and \eqref{E:support}. Let $u$ and $u_h$ be, respectively, the solutions to \eqref{E:WeakForm-NMS-Graph} and \eqref{E:WeakForm-discrete}. Then, 
\begin{equation*}
    \lim_{h \to 0} I_s[u_h] = I_s[u].
\end{equation*}
\end{Theorem}
\begin{proof}
Let $\interp$ be defined according to \eqref{E:interpolation}. Since 
$\interp u \in \mathbb{V}_h^g$, it follows that $I_s[\interp u] \geq I_s[u_h]$ and hence
\begin{equation*}
\begin{aligned}
    0 &\leq I_s[u_h] - I_s[u] \leq I_s[\interp u] - I_s[u] \\
    &= \iint_{Q_{\Omega}} \l( F_s\l(\frac{\interp u(x)-\interp u(y)}{|x-y|}\r)
        - F_s\l(\frac{u(x)-u(y)}{|x-y|}\r) \r) \frac{dxdy}{|x-y|^{d+2s-1}}.
\end{aligned}
\end{equation*}
Because $F_s$ is Lipschitz continuous, we deduce
\begin{equation*}
\begin{aligned}
0 \le   I_s[u_h] - I_s[u]  &\lesssim \iint_{Q_{\Omega}} 
    \frac{\l| (\interp u -u)(x) - (\interp u -u)(y) \r|}{|x-y|^{d+2s}} \; dxdy,
\end{aligned}
\end{equation*}
and using \Cref {prop:interpolation_estimate} we conclude that $\lim_{h \to 0} I_s[u_h] = I_s[u]$.
\end{proof}

Finally, we prove the convergence of the finite element approximations to the nonlocal minimal graph as the maximum element size $h \to 0$.

\begin{Theorem}[convergence] 
Under the same hypothesis as in \Cref{thm:consistency}, it holds that
\[
\lim_{h \to 0} \| u - u_h \|_{W^{2r}_1(\Omega)} = 0 \quad \forall r \in [0,s).
\]
\end{Theorem}
\begin{proof}
Due to our assumptions on $g$, we apply Theorem \ref{thm:consistency} to deduce that the finite element discretization is energy-consistent. Thus, the family $\{I_s[u_h]\}_{h>0}$ is uniformly bounded.

Similarly to the first formula in \eqref{eq:norm_bound}, we obtain
\begin{equation*}
|u_h|_{W^{2s}_1(\Omega)} \le C_1 + C_2 I_s[u_h] ,
\end{equation*}
whereas $\Vert u_h \Vert_{L^1(\Omega)}$ is bounded as in \eqref{eq:L1-bound}. It thus follows that $\Vert u_h \Vert_{W^{2s}_1(\Omega)}$ is uniformly bounded.

This fact, combined with the compactness of the embedding $W^{2s}_1(\Omega) \subset L^1(\Omega)$, allows us to extract a subsequence $\{u_{h_n}\}$, which converges to some $\widetilde{u}$ in $L^1(\Omega)$. According to \eqref{E:OmegaC-ae-converge} in the proof of \Cref{prop:interpolation_estimate}, $\{u_{h_n}\}$ converges to $g$ a.e. in $\Omega^c$; then, extending $\widetilde{u}$ as $g$ onto $\Omega^c$, we have by Fatou's Lemma that
\begin{equation*}
I_s[\widetilde{u}] \leq \liminf_{n \to \infty} I_s[u_{h_n}] = I_s[u].
\end{equation*}
As a consequence, $I_s[\widetilde{u}]$ is finite and, by \Cref{Prop:domain-energy}, it follows that $\widetilde{u} \in \mathbb{V}^g$. Because $\widetilde{u}$ minimizes the energy $I_s$,
it is a solution of \eqref{E:WeakForm-discrete},
and by uniqueness it must be $\widetilde{u} = u$. Since any subsequence of $\{ u_h \}$
has a subsequence converging to $u$ in $L^1(\Omega)$, it follows immediately that  $u_h$ converges to $u$ in $L^1(\Omega)$ as $h \to 0$. 

Finally, the convergence in the $W^{2r}_1(\Omega)$-norm for $r \in (0,s)$ is obtained by
interpolation between the spaces $L^1(\Omega)$ and $W^{2s}_1(\Omega)$.
\end{proof}

%%%%%%%%%%%%%%%%%%%%%%%%%%%%
%%%%%%%%%%%%%%%%%%%%%%%%%%%%
\section{A geometric notion of error} \label{sec:geometric-error}
%%%%%%%%%%%%%%%%%%%%%%%%%%%%
%%%%%%%%%%%%%%%%%%%%%%%%%%%%
In this section, we introduce a geometric notion of error and prove the convergence of the discrete approximations proposed in \Cref{ss:FE_discretization} according to it. The error estimate for this novel quantity mimics the estimates in the classical Plateau problem for the error
\begin{equation}\label{eq:def-e}
\begin{aligned}
e^2(u,u_h) & =\int_{\Omega} \ \Big| \widehat{\nu}(\nabla u) - \widehat{\nu}(\nabla u_h) \Big|^2 \;\frac{Q(\nabla u) + Q(\nabla u_h)}{2} \ dx , \\
& = \int_{\Omega} \ \Big( \widehat{\nu}(\nabla u) - \widehat{\nu}(\nabla u_h) \Big) \cdot \ \Big( \nabla (u-u_h), 0 \Big) dx ,
\end{aligned}
\end{equation}
where $Q(\pmb{a}) = \sqrt{1+|\pmb{a}|^2}$, $\widehat{\nu}(\pmb{a}) = \frac{(\pmb{a},-1)}{Q(\pmb{a})}$. Since, in this context, $\widehat{\nu}(\nabla u)$ is the normal unit vector on the graph of $u$, $e(u,u_h)$ represents a weighted $L^2$-error for the normal vectors of the corresponding graphs given by $u$ and $u_h$, where the weight is the average of the area
elements of the graphs of $u$ and $u_h$. An estimate for $e(u,u_h)$ was derived by Fierro and Veeser \cite{FierroVeeser03} in the framework of a posteriori error estimation for prescribed mean curvature equations. Geometric notions of errors like $e(u,u_h)$ have also been considered in the setting of mean curvature flows \cite{DeckelnickDziuk00, DeDzEl05} and surface diffusion \cite{BaMoNo04}.

For the nonlocal minimal surface problem, let $u$ and $u_h$ be the solutions to \eqref{E:WeakForm-NMS-Graph} and \eqref{E:WeakForm-discrete}, respectively. We introduce the quantity
\begin{equation}\label{eq:def-es} 
{
\begin{aligned}
e_s(u,u_h) := \l( C_{d,s} \iint_{Q_{\Omega}} \Big( G_s\l(d_u(x,y)\r) - G_s\l(d_{u_h}(x,y)\r) \Big) \frac{d_{u-u_h}(x,y)}{|x-y|^{d-1+2s}} dxdy \r)^{1/2} ,
\end{aligned}}
\end{equation}
where $G_s$ is given by \eqref{E:DEF-Gs}, the constant $C_{d,s} = \frac{1 - 2s}{\alpha_{d}}$, with $\alpha_{d}$ denoting the volume of the $d$-dimensional unit ball and, for any function $v$, $d_v(x,y)$ is defined as
\begin{equation}\label{eq:def-d_u}
d_v(x,y) := \frac{v(x)-v(y)}{|x-y|}.
\end{equation}
The term in parenthesis in \eqref{eq:def-es} is non-negative because $G_s$ is non-decreasing on $\mathbb{R}$. We include the constant $C_{d,s}$ in the definition of $e_s$ in order to have asymptotic compatibility in the limit $s \to \frac12^-$ (cf. \Cref{Thm:asymptotics-es} below). 

\Cref{ss:error} derives an estimate for $e_s(u,u_h)$ that does not rely on regularity assumptions. Although the proof of such an error estimate is simple, providing an interpretation of the quantity $e_s$ is not a straightforward task. Thus, in \Cref{ss:asymptotics} we study the behavior of $e_s$ and related quantities in the limit $s \to \frac12^-$.

%%%%%%%%%%%%%%%%%%%%%%%%%%%%
\subsection{Error estimate} \label{ss:error}
%%%%%%%%%%%%%%%%%%%%%%%%%%%%
In this section we derive an upper bound for the geometric discrepancy $e_s(u,u_h)$ between the continuous and discrete minimizers $u$ and $u_h$, without additional assumptions on the regularity of $u$. More precisely, the next theorem states that $e_s(u,u_h)$ can be bounded in terms of the approximability of $u$ by the discrete spaces $\mathbb{V}^g_h$ in terms of the $\mathbb{V}^g$-seminorm.

\begin{Theorem}[geometric error]\label{thm:geometric-error}
Let $s \in (0,1/2)$ and let $\Omega$ and $g$ satisfy \eqref{E:assumptions} and \eqref{E:support}. Let $u$ and $u_h$ be the solutions to \eqref{E:WeakForm-NMS-Graph} and \eqref{E:WeakForm-discrete} respectively. Then, it holds that
\begin{equation} \label{eq:geometric-error} \begin{aligned}
e_s(u,u_h) &\le \inf_{v_h \in \mathbb{V}_h^g} \sqrt{2 C_{d,s} K} \ | u-v_h |_{\mathbb{V}^g} \\
& = \inf_{v_h \in \mathbb{V}_h^g}  \l( 2 C_{d,s} K \iint_{Q_{\Omega}} \frac{|(u-v_h)(x)-(u-v_h)(y)|}{|x-y|^{d+2s}} dxdy \r) ^{1/2},
\end{aligned} \end{equation}
where $K$ is the constant from \eqref{E:bound-Gs}.
\end{Theorem}
\begin{proof}
The proof follows by `Galerkin orthogonality'. Indeed, let $v_h \in \mathbb{V}_h^g$ and use $u_h - v_h$ as test function in \eqref{E:WeakForm-NMS-Graph} and \eqref{E:WeakForm-discrete} to obtain
\[
\iint_{Q_{\Omega}} \Big( G_s\l(d_u(x,y)\r) - G_s\l(d_{u_h}(x,y)\r) \Big) \; \frac{d_{u_h}(x,y) - d_{v_h}(x,y)}{|x-y|^{d-1+2s}} dxdy = 0.
\]
The identity above immediately implies that
\begin{equation} \label{eq:orthogonality-es} \begin{aligned}
e^2_s(u,u_h)
&= C_{d,s} \iint_{Q_{\Omega}} \Big( G_s\l(d_u(x,y)\r) - G_s\l(d_{u_h}(x,y)\r) \Big) \frac{d_u(x,y) - d_{u_h}(x,y)}{|x-y|^{d-1+2s}} dxdy \\
&= C_{d,s} \iint_{Q_{\Omega}} \Big( G_s\l(d_u(x,y)\r) - G_s\l(d_{u_h}(x,y)\r) \Big) \frac{d_u(x,y) - d_{v_h}(x,y)}{|x-y|^{d-1+2s}} dxdy.
\end{aligned}
\end{equation} 
Estimate \eqref{eq:geometric-error} follows immediately from the bound $|G_s| \le K$ (cf. \eqref{E:bound-Gs}).
\end{proof}

In case the fractional minimal graph possesses additional regularity, a convergence rate follows straightforwardly by applying \Cref{prop:interpolation_estimate-2}.

\begin{Corollary}[convergence rate]
Let the same conditions as in \Cref{thm:geometric-error} be valid and further assume that $u \in W^t_1(\Lambda)$ for some $t \in (2s,2]$, where $\Lambda$ is given by \eqref{E:support}. Then,
\[
e_s(u,u_h) \lesssim h^{t/2-s} |u|^{1/2}_{W^t_1(\Lambda)} .
\]
\end{Corollary}

\begin{Remark}[BV regularity]
Although minimal graphs are expected to be discontinuous across the boundary, they are smooth in the interior of $\Omega$ and, naturally, possess the same regularity as the datum $g$ over $\Omega^c$. Therefore, in general, we expect that $u \in BV(\Lambda)$ whence the error estimate
\[
e_s(u,u_h) \lesssim h^{1/2-s} |u|^{1/2}_{BV(\Lambda)}.
\]
\end{Remark}

%%%%%%%%%%%%%%%%%%%%%%%%%%%%
\subsection{Asymptotic behavior} \label{ss:asymptotics}
%%%%%%%%%%%%%%%%%%%%%%%%%%%%
Our goal in this section is to show that,  for $u$ and $v$ smooth enough, $e_s(u,v)$ converges to the geometric notion of error $e(u,v)$ defined in \eqref{eq:def-e} in the limit $s \to {\frac{1}{2}}^-$. To this aim, we first introduce a nonlocal normal vector. 

\begin{Definition}[nonlocal normal vector] \label{def:normal-vector}
Let $s \in (0, 1/2)$ and $E \subset \mRd$ be an open, bounded, measurable set. The nonlocal inward normal vector of order $s$ at a point $x \in \partial E$ is defined as
\begin{equation}\label{eq:def-nonlocal-normal}
    \nu_s(x;E) = \frac{C_{d-1,s}}{2} \; \lim_{R \to \infty} \int_{B_R(x)}
     \frac{\chi_E(y) - \chi_{E^c}(y)}{|x-y|^{d+2s}} \; (y - x) \; dy,
\end{equation}
where $C_{d-1,s} = \frac{1-2s}{\alpha_{d-1}}$ as in \eqref{eq:def-es}, except that $d$ is replaced by $d-1$.
\end{Definition}

\begin{Remark}[dimensions]
We point out that, definition \eqref{eq:def-es} aims to measure the normal vector discrepancies over graphs in $\mRdp$, whereas \Cref{def:normal-vector} deals with the normal vector to a subset of $\mRd$. This is why in \eqref{eq:def-nonlocal-normal} we use the constant $C_{d-1,s}$ instead of $C_{d,s}$.
\end{Remark}

Notice that, by symmetry, 
\begin{equation*}
\int_{\partial B_R(x)} \frac{y - x}{|x-y|^{d+2s}} dS(y) = 0 \quad \forall R>0.
\end{equation*}
Consequently, because $\chi_{E^c} = 1 - \chi_{E}$, if $E \subset B_R(x)$ for some $R > 0$, then
\begin{equation*} \begin{aligned}
    \nu_s(x;E) & = \frac{C_{d-1,s}}{2} \int_{B_R(x)}
     \frac{\chi_E(y) - \chi_{E^c}(y)}{|x-y|^{d+2s}} (y - x) \; dy \\
     & = C_{d-1,s} \int_{B_R(x)}
     \frac{\chi_E(y)}{|x-y|^{d+2s}} (y - x) \; dy.
\end{aligned} \end{equation*}

The following lemma justifies that the nonlocal normal vector defined in \eqref{eq:def-nonlocal-normal} is indeed an extension of the classical notion of normal vector. The scaling factor in the definition of $\nu_s$ yields the convergence to the normal derivative as $s \to \frac12^-$.

\begin{Lemma}[asymptotic behavior of $\nu_s$]\label{lem:asymp-nonlocal-normal}
Let $E$ be a bounded set in $\mRd$, $x$ be a point on $\partial E$, the surface $\partial E$ be locally $C^{1,\gamma}$ for some $\gamma > 0$ and $\nu(x)$ be the inward normal vector to $\partial E$ at $x$. Then, the following holds:
\begin{equation}\label{eq:asymp-nonlocal-normal}
\lim_{s \to {\frac{1}{2}}^-}  \nu_s(x;E) = \nu(x) .
\end{equation}
\end{Lemma}
\begin{proof}
Without loss of generality, we assume $x = 0$. Let $\widetilde{E} := \{y: y \cdot \nu(x) \ge 0\}$ and for simplicity we write $B_r = B_r(x)$. Then, since $\partial E$ is locally $C^{1,\gamma}$, there exists some $r_0 > 0$ such that
\begin{equation}\label{eq:proof-asymp-nonlocal-normal}
 \left| \int_{\partial B_r} \chi_{E \triangle \widetilde{E}}(y) \; dS(y) \right| \lesssim r^{d+\gamma-1}
\end{equation}
for any $r \in (0,r_0]$, where $\triangle$ denotes the symmetric difference between sets. Fix $R > r_0$ large enough so that $E \subset B_R(x)$. Then, we can write $\nu_s(x;E)$ as
\begin{equation}\label{eq:proof-asymp-nonlocal-normal-2}
\begin{aligned}
\nu_s(x;E) &= C_{d-1,s} \int_{B_R} \frac{\chi_E(y)}{|y|^{d+2s}} \; y \; dy \\
&= C_{d-1,s} \l(\int_{B_R \setminus B_{r_0}} \frac{\chi_E(y)}{|y|^{d+2s}} \; y \; dy + \int_{B_{r_0}} \frac{\chi_E(y)}{|y|^{d+2s}} \; y \; dy \r).
\end{aligned}
\end{equation}
For the first integral in the right hand side, since the surface area of the $(d-1)$-dimensional unit ball equals $d \alpha_d$, we have
\[ \begin{aligned}
\l| \int_{B_R \setminus B_{r_0}} \frac{\chi_E(y)}{|y|^{d+2s}} \; y \; dy \r| & = \left| \int_{r_0}^R dr \int_{\partial B_r} \frac{\chi_E(y)}{r^{d+2s}} \; y \; dS(y) \right| \\
& \le \int_{r_0}^R dr \int_{\partial B_r} \frac{1}{r^{d+2s}} \; r^d \; dS = d \alpha_d \int_{r_0}^R  r^{-2s} \; dr \\
& =  \frac{d \, \alpha_d}{1-2s} (R^{1-2s} - r_0^{1-2s}).
\end{aligned} \]
Therefore, in the limit $s\to\frac12^-$, we obtain 
\begin{equation} \label{eq:lim-E-far}
C_{d-1,s} \l| \int_{B_R \setminus B_{r_0}} \frac{\chi_E(y)}{|y|^{d+2s}} \; y \; dy \r| \le
\frac{d \alpha_d}{\alpha_{d-1}} \l(R^{1-2s} - r_0^{1-2s} \r) \to 0.
\end{equation}

We now deal with the second term in the right hand side in \eqref{eq:proof-asymp-nonlocal-normal-2}. Without loss of generality, we additionally assume $\nu(x) = e_1$.
If we replace $E$ by the set $\widetilde{E}$ defined above, that coincides with the half-space $\{ y \colon y_1 \ge 0 \}$, it follows by symmetry that all components but the first one in the integral vanish. The first component can be calculated explicitly by writing it as an iterated integral along the $(d-1)$-dimensional slices $\Pi_t = \{ y_1 = t \}$ and integrating in polar coordinates on these:
\[ \begin{aligned}
\int_{B_{r_0}} \frac{\chi_{\widetilde{E}}(y)}{|y|^{d+2s}} \; y_1 \; dy & =
\int_0^{r_0} dt \int_{\Pi_t \cap \{ |z|^2 \le r_0^2 - t^2\} } \frac{t}{\l( t^2 + |z|^2 \r)^\frac{d+2s}{2}} \; dz \\
& = \int_0^{r_0} dt \int_0^{\sqrt{r_0^2 - t^2}} dr \int_{\partial B^{(d-2)}_r} \frac{t}{\l( t^2 + r^2 \r)^\frac{d+2s}{2}} \; dS(z) \\
& = (d-1) \alpha_{d-1} \int_0^{r_0} dt \int_0^{\sqrt{r_0^2 - t^2}} \frac{t r^{d-2}}{\l( t^2 + r^2 \r)^\frac{d+2s}{2}} \; dr .
\end{aligned} \]

The iterated integral above can be calculated with elementary manipulations (Fubini's theorem, change of variables $t \mapsto w = \l(\frac{t}{r}\r)^2$ and explicit computation of integrals) to give
\[
\int_0^{r_0} dt \int_0^{\sqrt{r_0^2 - t^2}} \frac{t r^{d-2}}{\l( t^2 + r^2 \r)^\frac{d+2s}{2}} \; dr = \frac{r_0^{1-2s}}{(d-1)(1-2s)},
\]
and therefore, as $s \to \frac12^-$,
\[
C_{d-1,s} \int_{B_{r_0}} \frac{\chi_{\widetilde{E}}(y)}{|y|^{d+2s}} \; y_1 \; dy = C_{d-1,s} \; (d-1) \alpha_{d-1} \frac{r_0^{1-2s}}{(d-1)(1-2s)} = r_0^{1-2s} \to 1.
\]
This shows that
\begin{equation} \label{eq:lim-E-tilde}
\lim_{s \to \frac12^-} \int_{B_{r_0}} \frac{2\chi_{\widetilde{E}}(y)}{|y|^{d+2s}} \; y_1 \; dy = \nu(x) .
\end{equation} 

Using \eqref{eq:proof-asymp-nonlocal-normal}, the difference between the integrals over $E$ and $\widetilde{E}$ can be bounded as
\[ \begin{aligned}
C_{d-1,s} \left| \int_{B_{r_0}} \frac{\chi_{E} (y) - \chi_{\widetilde{E}}(y)}{|y|^{d+2s}} \; y \; dy  \right|
& \le C_{d-1,s} \int_0^{r_0} dr \int_{\partial B_r} \chi_{E \triangle \widetilde{E}}(y) \; r^{-d-2s} \; r \; dS(y) \\
& \lesssim C_{d-1,s} \int_0^{r_0} r^{\gamma-2s} \;  dr = C_{d-1,s}  \; \frac{r_0^{\gamma+1-2s}}{\gamma+1-2s},
\end{aligned} \]
where the right hand side above tends to $0$ because $C(d-1,s) = \frac{1-2s}{\alpha_{d-1}}$ and $\gamma > 0$ is fixed.  Combining this estimate with \eqref{eq:proof-asymp-nonlocal-normal-2}, \eqref{eq:lim-E-far} and \eqref{eq:lim-E-tilde}, we finally get
\[
\lim_{s \to {\frac{1}{2}}^-} \nu_s(x;E)
= \nu(x),
\]
thereby finishing the proof.
\end{proof}

\begin{Remark}[localization]\label{rem:asymp-nonlocal-normal}
From the preceding proof, it follows that only the part of the integral near $x$ remains in the limit when $s \to {\frac{1}{2}}^-$. 
Thus, for any neighborhood $\mathcal{N}_x$ of $x$, we could similarly prove
\[
\lim_{s \to {\frac{1}{2}}^-} \frac{C_{d-1,s}}{2} \; \int_{\mathcal{N}_x} \frac{\chi_E(y) - \chi_{E^c}(y)}{|x-y|^{d+2s}} \; (y - x) \; dy = \nu(x)
\]
without the assumption of the boundedness of $E$.
\end{Remark}

We now go to the graph setting and consider 
\[
E = \left\{ (x, x_{d+1}): x_{d+1} \le u(x), x \in \mRd \right\} \subset \mRdp,
\] 
where $u \in L^{\infty}(\mRd)$. For such a set $E$ it is clear that our definition \eqref{eq:def-nonlocal-normal} is not adequate: the limit of the integral therein does not exist. However, the only issue in such a definition is that the last component of the nonlocal normal vector in $\mRdp$ tends to $-\infty$, and thus it can be solved in a simple way. Indeed, we introduce the projection operator $P$ that maps 
\[
\mRdp \ni \widetilde{x} = (x,x_{d+1}) \mapsto P(\widetilde{x}) = x \in \mRd.
\] 
Then we could actually define the normal vector for this type of unbounded set $E$ as the projection $P\left(\nu_s(x;E)\right)$. 

More precisely, given $\widetilde{x} = \left(x,u(x) \right)$, we define the projection of nonlocal normal vector, $\widetilde{\nu}_s(\widetilde{x};E) = P\left(\nu_s(\widetilde{x};E)\right)$, as
\begin{equation}\label{eq:def-nonlocal-normal-proj}
    \widetilde{\nu}_s(\widetilde{x};E) = \frac{C_{d,s}}{2} \lim_{R \to \infty} \int_{B_R(\widetilde{x})}
     \frac{\chi_E(\widetilde{y}) - \chi_{E^c}(\widetilde{y})}{|\widetilde{x}-\widetilde{y}|^{d+1+2s}} \; P(\widetilde{y} - \widetilde{x}) \; \dty,
\end{equation}
where $x = P(\widetilde{x})$ and $y = P(\widetilde{y})$.
To show that this limit exists, consider the sets
\[ \begin{aligned}
& B_R^+\left(\widetilde{x}\right) := \l\{ \widetilde{y} = (y,y_{d+1}) \in \mRdp : |\widetilde{y} - \widetilde{x}| \le R, \; y_{d+1} \ge u(x) \r\}, \\
& B_R^-(\widetilde{x}) := B_R(\widetilde{x}) \setminus B_R^+(\widetilde{x}).
\end{aligned}\]
Since both $B_R^+\l(\widetilde{x}\r)$ and $B_R^-\l(\widetilde{x}\r)$ are half balls, by  symmetric cancellation, we have
\[
\int_{B_R^+(\widetilde{x})}  \frac{1}{|\widetilde{x}-\widetilde{y}|^{d+1+2s}} \dty  - \int_{B_R^-(\widetilde{x})}  \frac{1}{|\widetilde{x}-\widetilde{y}|^{d+1+2s}} \dty  = 0
\]
in the principal value sense. Therefore, using that $\chi_{E} = 1 - \chi_{E^c}$, we can express
\[\begin{aligned}
\int_{B_R(\widetilde{x})} \frac{\chi_E(\widetilde{y}) - \chi_{E^c}(\widetilde{y})}{|\widetilde{x}-\widetilde{y}|^{d+1+2s}} & \; P(\widetilde{y} - \widetilde{x}) \; \dty \\
     & = 2\int_{B_R^+(\widetilde{x}) \cap E}
     \frac{(y - x)}{|\widetilde{x}-\widetilde{y}|^{d+1+2s}} \dty - 2\int_{B_R^-(\widetilde{x}) \setminus E}
     \frac{(y - x)}{|\widetilde{x}-\widetilde{y}|^{d+1+2s}} \dty.
\end{aligned} \]
The two integrands above have enough decay at infinity because we are assuming $u \in L^\infty(\mRd)$. Thus, as $R \to \infty$, we may replace $B_R^-(\widetilde{x})$ by the  half space $H^-\l(\widetilde{x}\r) :=  \{(y,y_{d+1}) \in \mRdp : y_{d+1} < u(x)\}$. Thus, the vector defined in \eqref{eq:def-nonlocal-normal-proj} can be written as
\[ \begin{aligned}
\widetilde{\nu}_s(\widetilde{x};E) &= C_{d,s} \int_{E \setminus H^-\l(\widetilde{x}\r)}
     \frac{(y-x)}{|\widetilde{x}-\widetilde{y}|^{d+1+2s}} \dty - C_{d,s} \int_{H^-\l(\widetilde{x}\r) \setminus E}
     \frac{(y-x)}{|\widetilde{x}-\widetilde{y}|^{d+1+2s}} \dty \\
     &= C_{d,s} \int_{\mRd} dy \int_{u(x)}^{u(y)} \frac{(y-x)}{\left(|x-y|^2+(y_{d+1}-u(x))^2 \right)^{\frac{d+1+2s}{2}}} \; dy_{d+1} .
\end{aligned} \]

Making the substitution $t = \frac{y_{d+1}-u(x)}{|x-y|}$, recalling the definitions of $G_s$ and $d_u$ (cf. \eqref{E:DEF-Gs} and \eqref{eq:def-d_u}, respectively), and noticing that $d_u(x,y) = -d_u(y,x)$, we conclude that
\[ \begin{aligned}
\widetilde{\nu}_s(\widetilde{x};E) & = C_{d,s} \int_{\mRd} (y-x) \; dy \int_{0}^{\frac{u(y)-u(x)}{|x-y|}} \frac{1}{|x-y|^{d+2s} \left( 1+t^2 \right)^{\frac{d+1+2s}{2}}} \; dt  \\
     &= C_{d,s} \int_{\mRd} \frac{G_s \left(d_u(x,y) \right)}{|x-y|^{d+1+2s}} \; (x-y) \; dy.
\end{aligned} \]

As we mentioned above, $\widetilde{\nu}_s(\widetilde{x};E)$ can be regarded as the projection of $\nu_s(\widetilde{x};E)$ under $P$. Therefore, following similar steps as in \Cref{lem:asymp-nonlocal-normal} and \Cref{rem:asymp-nonlocal-normal}, it is possible to prove the following result.

\begin{Lemma}[asymptotics of $\widetilde{\nu}_s$]\label{lem:asymp-nonlocal-normal-proj}
Let $E = \left\{ (x, x_{d+1}): x_{d+1} \le u(x), x \in \mRd \right\}$, where $u \in L^{\infty}(\mRd)$ and $u$ is locally $C^{1,\gamma}$ around a point $x$ for some $\gamma > 0$. Then, the following asymptotic behavior holds
\begin{equation}\label{eq:asymp-nonlocal-normal-proj}
\begin{aligned}
\lim_{s \to {\frac{1}{2}}^-} \widetilde{\nu}_s(\widetilde{x};E) &= 
\lim_{s \to {\frac{1}{2}}^-} C_{d,s} \int_{\mRd} \frac{G_s \left(d_u(x,y) \right)}{|x-y|^{d+2s}} \; (x-y) \; dy \\
&= \frac{\nabla u(x)}{\sqrt{1+|\nabla u(x)|^2}},
\end{aligned}
\end{equation}
where $\widetilde{x} = (x,u(x))$. In addition, we also have
\begin{equation}\label{eq:asymp-nonlocal-normal-proj2}
\lim_{s \to {\frac{1}{2}}^-} C_{d,s} \int_{\mathcal{N}_x} \frac{G_s \left(d_u(x,y) \right)}{|x-y|^{d+2s}} \; (x-y) \; dy
= \frac{\nabla u(x)}{\sqrt{1+|\nabla u(x)|^2}},
\end{equation}
for any neighborhood $\mathcal{N}_x$ of $x$.
\end{Lemma}

Our next lemma deals with the interaction between the nonlocal normal vector to the graph of $u \colon \mRd \to \mR$ and a function $v\colon \mRd \to \mR$. For that purpose, we redefine $a_u$ so as to include the proper scaling factor for $s \to \frac12^-$.
Indeed, given $u \in \mathbb{V}^g$, we set $a_u \colon \mathbb{V}^g \times \mathbb{V}^0 \to \mathbb{R}$ to be
\begin{equation} \label{E:def-a-scaled}
a_u(w,v) := C_{d,s} \iint_{Q_{\Omega}} \Gts\l(\frac{u(x)-u(y)}{|x-y|}\r) \frac{(w(x)-w(y))(v(x)-v(y))}{|x-y|^{d+1+2s}}dx dy.
\end{equation}

\begin{Lemma}[asymptotics of $a_u$ with H\"older regularity]\label{lem:asymp-nonlocal-normal-interaction}
Let $u,v \in C^{1,\gamma}_c(\Lambda)$ for some $\gamma > 0$ and a bounded set $\Lambda$ containing $\Omega \subset \mRd$. Then, it holds that
\[
\lim_{s \to {\frac{1}{2}}^-}
a_{u}(u,v) = \int_{\Omega} \frac{\nabla u(x) \cdot \nabla v(x)}{\sqrt{1+|\nabla u(x)|^2}} dx,
\]
where $a_u(u,v)$ is the form defined in \eqref{E:def-a-scaled}.
\end{Lemma}

\begin{Remark}[heuristic interpretation of \Cref{lem:asymp-nonlocal-normal-interaction}]
Suppose $v$ was a linear function. Then, for all $x,y$ we have $v(x) - v(y) = (x-y) \cdot \nabla v(x)$, and thus we can write
\[
a_u(u,v) = C_{d,s} \iint_{Q_{\Omega}} G_s\l(d_u(x,y) \r) \frac{(x-y) \cdot \nabla v(x)}{|x-y|^{d+2s}} \;  dx dy,
\] 
while
\[
\widetilde{\nu}_s(\widetilde{x};E) \cdot \nabla v (x) = 
C_{d,s} \left(\int_{\mRd} \frac{G_s \left(d_u(x,y) \right)}{|x-y|^{d+2s}} (x-y) \; dy \right) \cdot \nabla v (x) .
\]
Therefore, taking into account the asymptotic behavior in \eqref{eq:asymp-nonlocal-normal-proj2},
\Cref{lem:asymp-nonlocal-normal-interaction} would follow upon integration of the identity above over $\Omega$. 
However, for an arbitrary (nonlinear) $v$, we can only interpret $a_u(u,v)$ as a certain interaction between the nonlocal normal vector $\widetilde{\nu}_s$ and the `nonlocal gradient' $d_v$. Nevertheless, in the limit $s \to {\frac{1}{2}}^-$, only the interaction for $x,y$ close remains, and the asserted result follows because any $C^1$ function is locally linear.

\end{Remark}

\begin{proof}[Proof of \Cref{lem:asymp-nonlocal-normal-interaction}]
We first split the domain of integration using symmetry:
\begin{equation} \label{eq:proof-asymp-nonlocal-normal-interaction-1}
\begin{aligned}
a_u(u,v) &= C_{d,s} \iint_{\Omega \times \Omega} G_s\l(d_u(x,y)\r) \frac{d_v(x,y) }{|x-y|^{d-1+2s}} dxdy \\
&+2 C_{d,s} \iint_{\Omega \times \Omega^c} G_s\l(d_u(x,y)\r) \frac{d_v(x,y) }{|x-y|^{d-1+2s}} dxdy.
\end{aligned}\end{equation}

Consider the first integral in \eqref{eq:proof-asymp-nonlocal-normal-interaction-1}. For a fixed $x \in \Omega$, we expand $v(y) = v(x) + \nabla v(x) \cdot (x-y) + O(|x-y|^{1+\gamma})$ and exploit the fact that $|G_s|$ is uniformly bounded (cf. \eqref{E:bound-Gs}) to obtain
\begin{equation}\label{eq:proof-asymp-nonlocal-normal-interaction-2}
{\begin{aligned}
& \int_{\Omega} G_s\l(d_u(x,y)\r) \frac{d_v(x,y) }{|x-y|^{d-1+2s}} dy \\
& = \int_{\Omega} G_s\l(d_u(x,y)\r) \frac{\nabla v(x) \cdot (x-y) + O\l(|x-y|^{1+\gamma}\r) }{|x-y|^{d+2s}} dy \\
& = \left( \int_{\Omega} G_s\l(d_u(x,y)\r) \frac{x - y}{|x-y|^{d+2s}} dy \right) \cdot \nabla v(x) + 
O\left( \int_{\Omega} \frac{1}{|x-y|^{d+2s-1-\gamma}} dy \right).
\end{aligned}}
\end{equation}
Let us define $D = \sup_{x,y \in \Lambda} |x-y|$. Then, it is clear that $\Omega \subset \Lambda \subset \overline{B}_D(x)$ and integrating in polar coordinates we get
\[ \begin{aligned}
C_{d,s} \int_{\Omega} \frac{1}{|x-y|^{d+2s-1-\gamma}} dy & \lesssim \frac{1-2s}{\alpha_{d}} \int_0^D  r^{-2s+\gamma} dr \\% \to 0, \quad \text{ as } s \to {\frac{1}{2}}^-. %\\
& \lesssim \frac{1-2s}{(\gamma+1-2s) \alpha_{d}} D^{\gamma+1-2s} \to 0, \quad \text{ as } s \to {\frac{1}{2}}^-.
\end{aligned} \]
Identity \eqref{eq:asymp-nonlocal-normal-proj2} guarantees that
\[
\lim_{s \to {\frac{1}{2}}^-} C_{d,s} \left( \int_{\Omega} G_s\l(d_u(x,y)\r) \frac{x - y}{|x-y|^{d+2s}} dy \right) \cdot \nabla v(x) = \frac{\nabla u(x) \cdot \nabla v(x)}{\sqrt{1+|\nabla u(x)|^2}},
\]
so that it follows from \eqref{eq:proof-asymp-nonlocal-normal-interaction-2} that
\[
\lim_{s \to {\frac{1}{2}}^-} C_{d,s} \int_{\Omega} G_s\l(d_u(x,y)\r) \frac{d_v(x,y) }{|x-y|^{d-1+2s}} dy =
\frac{\nabla u(x) \cdot \nabla v(x)}{\sqrt{1+|\nabla u(x)|^2}},
\]
for every $x \in \Omega$. Since for all $x \in \Omega$ we have
\[ \begin{aligned}
\left| C_{d,s} \int_{\Omega} G_s\l(d_u(x,y)\r) \frac{d_v(x,y) }{|x-y|^{d-1+2s}} dy \right| & \le
C_{d,s} \, |v|_{C^{0,1}(\Oc)} \int_{\Oc} \frac{1}{|x-y|^{d-1+2s}} dy \\
& \le d (1-2s) \int_0^D r^{-2s} dr
= d D^{1-2s},
\end{aligned} \]
we can apply the Lebesgue Dominated Convergence Theorem to deduce that
\begin{equation} \label{E:convergence-Omega}
\lim_{s \to {\frac{1}{2}}^-} C_{d,s} \int_{\Omega\times \Omega} G_s\l(d_u(x,y)\r) \frac{d_v(x,y)}{|x-y|^{d-1+2s}} dxdy
= \int_{\Omega} \frac{\nabla u(x) \cdot \nabla v(x)}{\sqrt{1+|\nabla u(x)|^2}} dx.
\end{equation}

It remains to prove that the last term in \eqref{eq:proof-asymp-nonlocal-normal-interaction-1} converges to $0$ as $s \to {\frac{1}{2}}^-$. This is a consequence of the Dominated Convergence Theorem as well. For  $x \in \Omega$, we write $\delta(x) = \mbox{dist}(x,\partial\Omega)$. We first use that $|G_s|$ is uniformly bounded, according to \eqref{E:bound-Gs}, and integrate in polar coordinates to obtain, for every $x \in \Omega$,
\[
\l| C_{d,s} \int_{\Omega^c} G_s\l(d_u(x,y)\r) \frac{d_v(x,y) }{|x-y|^{d-1+2s}} dy \r|
\lesssim (1-2s) \; |v|_{C^{0,1}(\Oc)} \int_{\delta(x)}^D r^{-2s} \;dr \to 0.
\]

To prove that the integrands are uniformly bounded, we invoke again the uniform boundedness of $G_s$ and split the integral with respect to $y$ into two parts:
\[{\begin{aligned}
&  C_{d,s} \int_{\Omega^c} G_s\l(d_u(x,y)\r) \frac{d_v(x,y) }{|x-y|^{d-1+2s}} dy \lesssim  (1-2s) \int_{\Omega^c} \frac{d_v(x,y) }{|x-y|^{d-1+2s}} dy\\
& \lesssim (1-2s) \l( \int_{\{y \colon |y-x| \le 1 \}} \frac{d_v(x,y) }{|x-y|^{d-1+2s}} dy
+ \int_{\{ y \colon |y-x| > 1\} }\frac{d_v(x,y) }{|x-y|^{d-1+2s}} dy \r) \\
& \leq (1-2s) \l( \int_{\{y \colon |y-x| \le 1 \}} \frac{|v|_{C^{0,1}(\Lambda)} }{|x-y|^{d-1+2s}} dy + \int_{\{ y \colon |y-x| > 1\} } \frac{2|v|_{L^{\infty(\Lambda)}}}{|x-y|^{d+2s}} dy \r) \\
& \lesssim (1-2s) \l( \int_0^1 r^{-2s} dr + \int_1^{\infty} r^{-2s-1} dr \r) =  1 + \frac{1-2s}{2s}.
\end{aligned}} \]
Consequently, we have proved that
\[
\lim_{s \to {\frac{1}{2}}^-} C_{d,s} \iint_{\Omega\times \Omega^c} G_s\l(d_u(x,y)\r) \frac{d_v(x,y)}{|x-y|^{d-1+2s}} dxdy
= 0.
\]
This, together with \eqref{eq:proof-asymp-nonlocal-normal-interaction-1} and \eqref{E:convergence-Omega}, finishes the proof.
\end{proof}

Actually, the regularity assumptions in \Cref{lem:asymp-nonlocal-normal-interaction} can be weakened by a density argument. To this aim, we recall the following stability
result proved in \cite[Theorem 1]{BourBrezMiro2001another}: given $f \in W^1_p(\Omega)$, $1 \le p < \infty$ and $\rho \in L^1(\mRd)$ such that $\rho \ge 0$,
\begin{equation}\label{eq:Bourgain-stab}
\iint_{\Omega \times \Omega} \frac{|f(x)-f(y)|^p}{|x-y|^p} \rho(x-y) \;dxdy \le C \Vert f \Vert^p_{W^1_p(\Omega)} \Vert \rho \Vert_{L^1(\mRd)}.
\end{equation}
The constant $C$ depends only on $p$ and $\Omega$. We next state and prove a modified version of \Cref{lem:asymp-nonlocal-normal-interaction}.

\begin{Lemma}[asymptotics of $a_u$ with Sobolev regularity]\label{lem:asymp-nonlocal-normal-interaction-2}
Let $u,v \in H^1_0(\Lambda)$, for some bounded set $\Lambda$ containing $\Omega$. Then, it holds that
\[
\lim_{s \to {\frac{1}{2}}^-} a_u(u,v) = \int_{\Omega} \frac{\nabla u(x) \cdot \nabla v(x)}{\sqrt{1+|\nabla u(x)|^2}} dx.
\]
\end{Lemma}

\begin{proof}
First we point out that the double integral in the definition of $a_u$ is stable in the $H^1$-norm. More specifically, for $u_1,u_2,v_1,v_2 \in H^1_0(\Lambda)$, we have
\[ \begin{aligned}
 |a_{u_1}&(u_1,v_1)  - a_{u_2}(u_2,v_2)|
 \le  |a_{u_1}(u_1,v_1) - a_{u_1}(u_1,v_2) + a_{u_1}(u_1,v_2) - a_{u_2}(u_2,v_2) | \\
&\lesssim (1-2s) \iint_{\mRd \times \mRd} \frac{|d_{u_1}(x,y)| \; |d_{v_1-v_2}(x,y)| + |d_{u_1-u_2}(x,y)| \; |d_{v_2}(x,y)|}{|x-y|^{d-1+2s}} \; dx dy.
\end{aligned} \]

As before, set $D = \sup_{x,y \in \Lambda} |x-y|$. Using the Cauchy-Schwarz inequality and choosing
\[
\rho(x) = \left\{
\begin{array}{ll}
|x|^{-d+1-2s}, \quad & |x| \le D \\
0, \quad & |x| > D
\end{array}
\right.
\]
in \eqref{eq:Bourgain-stab}, we obtain that 
\[ \begin{aligned}
 |a_{u_1}& (u_1,v_1) - a_{u_2}(u_2,v_2)| \\
& \lesssim (1-2s) \left(|u_1|_{H^1(\mRd)} |v_1 - v_2|_{H^1(\mRd)} + |u_1-u_2|_{H^1(\mRd)} |v_2|_{H^1(\mRd)} \right)
\| \rho \|_{L^1(\mRd)}.
\end{aligned} \]

For the function $\rho$ we have chosen, it holds that 
\[
 \| \rho \|_{L^1(\mRd)} = \frac{d \alpha_d D^{1-2s}}{1-2s}.
\]
Thus, we obtain the following stability result for the form $a$:
\begin{equation} \label{E:stability-a}
{|a_{u_1}(u_1,v_1) - a_{u_2}(u_2,v_2)|
\le C \left(|u_1|_{H^1(\mRd)} |v_1 - v_2|_{H^1(\mRd)} + |u_1-u_2|_{H^1(\mRd)} |v_2|_{H^1(\mRd)} \right),}
\end{equation}
where the constant $C$ is independent of $s \in (0,\frac{1}{2})$ and the functions involved. 

A standard argument now allows us to conclude the proof. Given $u,v \in H^1_0(\Lambda)$, consider sequences $\{u_n\}, \{v_n\} \in C^{\infty}_c(\Lambda)$ such that $u_n \to u$ and $v_n \to v$ in $H^1(\mRd)$. Due to \Cref{lem:asymp-nonlocal-normal-interaction}, we have, for every $n$,
\[
\lim_{s \to {\frac{1}{2}}^-} a_{u_n}(u_n,v_n) = \int_{\Omega} \frac{\nabla u_n(x) \cdot \nabla v_n(x)}{\sqrt{1+|\nabla u_n(x)|^2}} dx.
\]
Applying \eqref{E:stability-a}, the claim follows.
\end{proof}

We are finally in position to show the asymptotic behavior of the notion of error $e_s$ introduced at the beginning of this section (cf. \eqref{eq:def-es}). Notice that, with the rescaling \eqref{E:def-a-scaled}, 
\[
e_s^2(u,v) = a_u(u,u) - a_u(u,v) - a_v(v,u) + a_v(v,v),
\]
while for its local counterpart \eqref{eq:def-e},
\[
e^2(u,v) = \int_\Omega \frac{\nabla u(x) \cdot \nabla u(x)}{\sqrt{1+|\nabla u(x)|^2}}
- \frac{\nabla u(x) \cdot \nabla v(x)}{\sqrt{1+|\nabla u(x)|^2}}
- \frac{\nabla v(x) \cdot \nabla u(x)}{\sqrt{1+|\nabla v(x)|^2}}
+ \frac{\nabla v(x) \cdot \nabla v(x)}{\sqrt{1+|\nabla v(x)|^2}}.
\]
Applying \Cref{lem:asymp-nonlocal-normal-interaction-2} term by term in the expansions above, we conclude that effectively, $e_s$ recovers $e$ in the limit. 

\begin{Theorem}[asymptotics of $e_s$] \label{Thm:asymptotics-es}
Let $u,v \in H^1_0(\Lambda)$, for some bounded set $\Lambda$ containing $\Omega$. Then, we have
\[
\lim_{s \to {\frac{1}{2}}^-} e_s(u,v) = e(u,v).
\]
\end{Theorem}

%%%%%%%%%%%%%%%%%%%%%%%%%%%%
%%%%%%%%%%%%%%%%%%%%%%%%%%%%
\section{Numerical experiments} \label{sec:numerics}
%%%%%%%%%%%%%%%%%%%%%%%%%%%%
%%%%%%%%%%%%%%%%%%%%%%%%%%%%
This section presents some numerical results that illustrate the properties of the algorithms discussed in \Cref{SS:schemes}. As an example, we consider $\Omega = B_1 \setminus \overline{B}_{1/2}$, where $B_r$ denotes an open ball with radius $r$ centered at the origin. For the Dirichlet data, we simply let $g = 0$ in $\mRd \setminus B_1$ and $g = 0.4$ in $\overline{B}_{1/2}$. Our computations are performed on an Intel Xeon E5-2630 v2 CPU (2.6 GHz), 16 GB RAM using MATLAB R2016b. More numerical experiments will be presented in an upcoming paper by the authors \cite{BoLiNo19computational}.

\begin{Remark}[classical minimal graph in a symmetric annulus]\label{R:NMS-classical-annulus}
We consider the classical graph Plateau problem in the same domain as our example above, with $g = 0$ on $\partial B_1$ and $g = \gamma$ on $\partial B_{1/2}$.
When $\gamma > \gamma^* := \frac{1}{2}\ln(2+\sqrt{3}) \approx 0.66$, the minimal surface consists of two parts. The first part is given by the graph of function $u(x,y) = \gamma^* - \frac{1}{2} \cosh^{-1}(2\sqrt{x^2+y^2})$ and the second part is given by $\{(x,y,z): \gamma^* \le z \le \gamma, (x,y) \in \partial B_{1/2} \}$. In this situation, a stickiness phenomenon occurs and $u$ is discontinuous across $\partial B_{1/2}$. Notice with our choice of Dirichlet data $\gamma = 0.4 < \gamma^*$, stickiness should not be observed for the classical minimal graph.
\end{Remark}

We first compute the solution $u_h$ of nonlinear system \eqref{E:WeakForm-discrete} using the $L^2$-gradient flow mentioned in \Cref{SS:schemes}. For $s = 0.25$ and mesh size $h = 2^{-4}$, we choose the initial solution $u_h^0 = 0$ and time step $\tau = 1$. The computed discrete solution $u_h$ is plotted in \Cref{F:GF-s025}. By symmetry we know the continuous solution $u$ should be radially symmetric, and we almost recover this property on the discrete level except in the region very close to $\partial B_{1/2}$ where the norm of $\nabla u_h$ is big. It is also seen that $0 \le u_h \le 0.4$, which shows computationally that the numerical scheme is stable in $L^{\infty}$ for this example. To justify convergence of the $L^2$-gradient flow, we consider the hat functions $\{ \varphi_i \}_{i=1}^N$ forming a basis of $\mathbb{V}^0_h$ where $N$ is the degrees of freedom. Consider residual vector $\{ r_i \}_{i=1}^N$ where $r_i := a_{u_h^k}(u_h^k, \varphi_i)$, we plot the Euclidean norm $\Vert r \Vert_{l^2}$ along the iteration for different time step $\tau$ in \Cref{F:GF-properties} (left). In the picture, the line for $\tau = 1$ and $\tau = 10$ almost coincide and we get faster convergence (fewer iterations) for larger time step $\tau$. For every choice of time step, we observe the linear convergence for the gradient flow iteration computationally. We have also tried different choices of initial solution $u_h^0$, and always end up observing the similar linear convergence behavior. 
We also plot the energy $I_s[u_h^k]$ along the iterations in \Cref{F:GF-properties} (right); this shows that the energy $I_s[u_h^k]$ monotonically decreases along the gradient flow iterations independently of the step size $\tau > 0$. This energy decay property will be proved in the upcoming paper \cite{BoLiNo19computational}.

\begin{figure}[!htb]
	\begin{center}
		\includegraphics[width=0.425\linewidth]{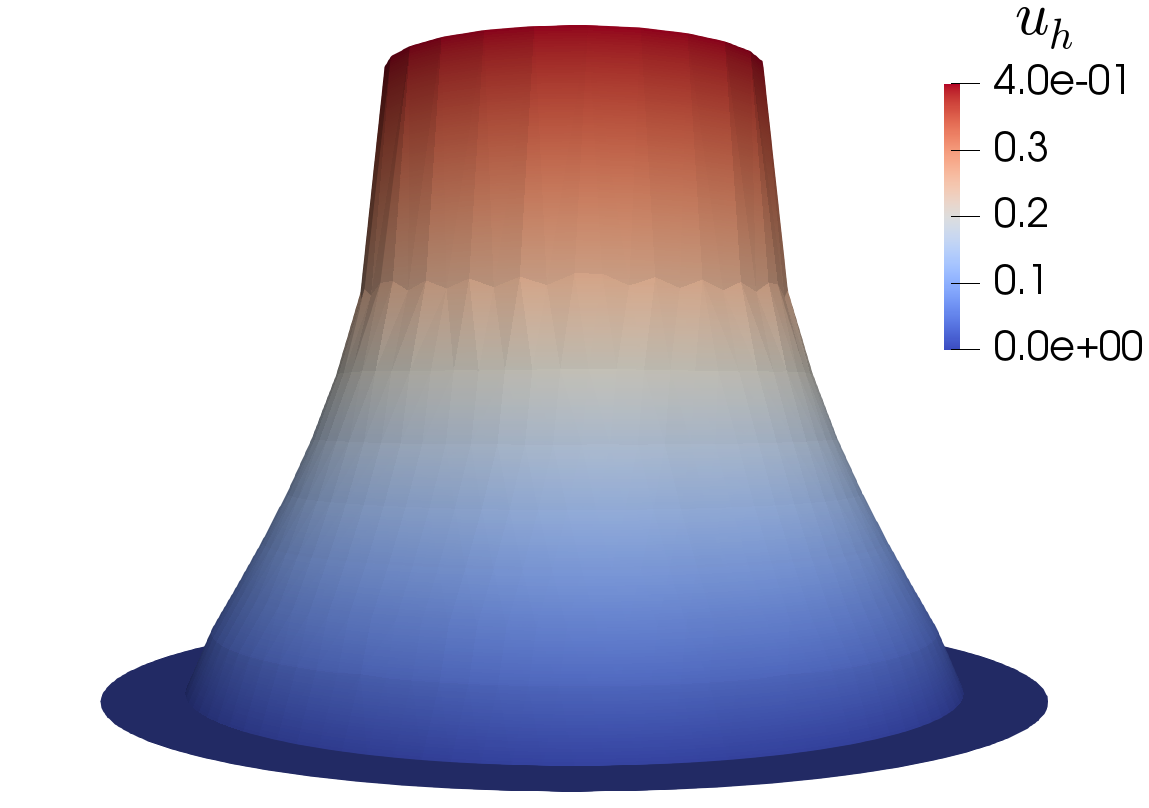}
	\end{center}
	\caption{\small Plot of $u_h$ computed by $L^2$-gradient flow for $s = 0.25$ and $h = 2^{-4}$.}
	\label{F:GF-s025}
\end{figure}

\begin{figure}[!htb]
	\begin{center}
		\includegraphics[width=0.376\linewidth]{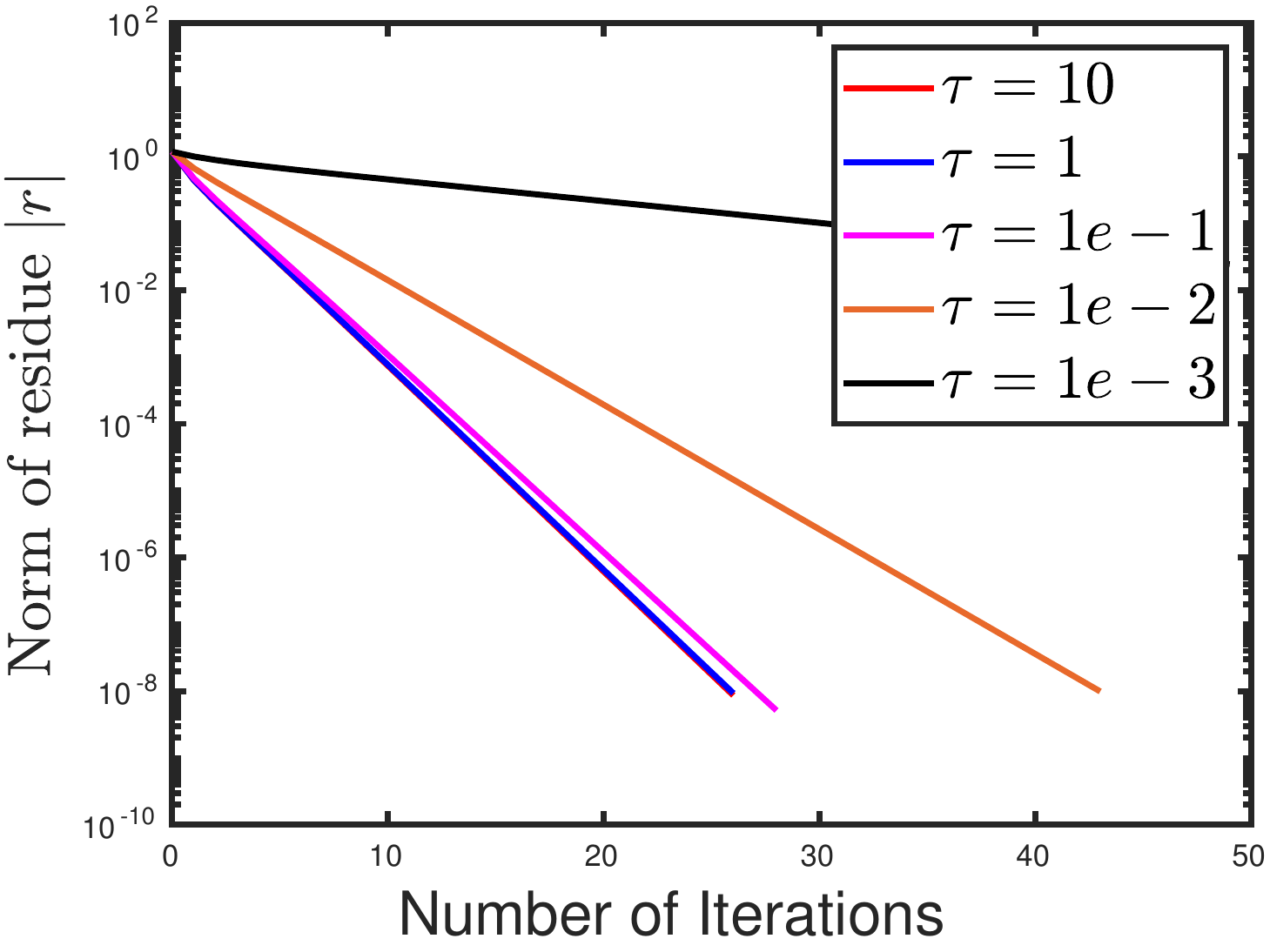}
		\includegraphics[width=0.38\linewidth]{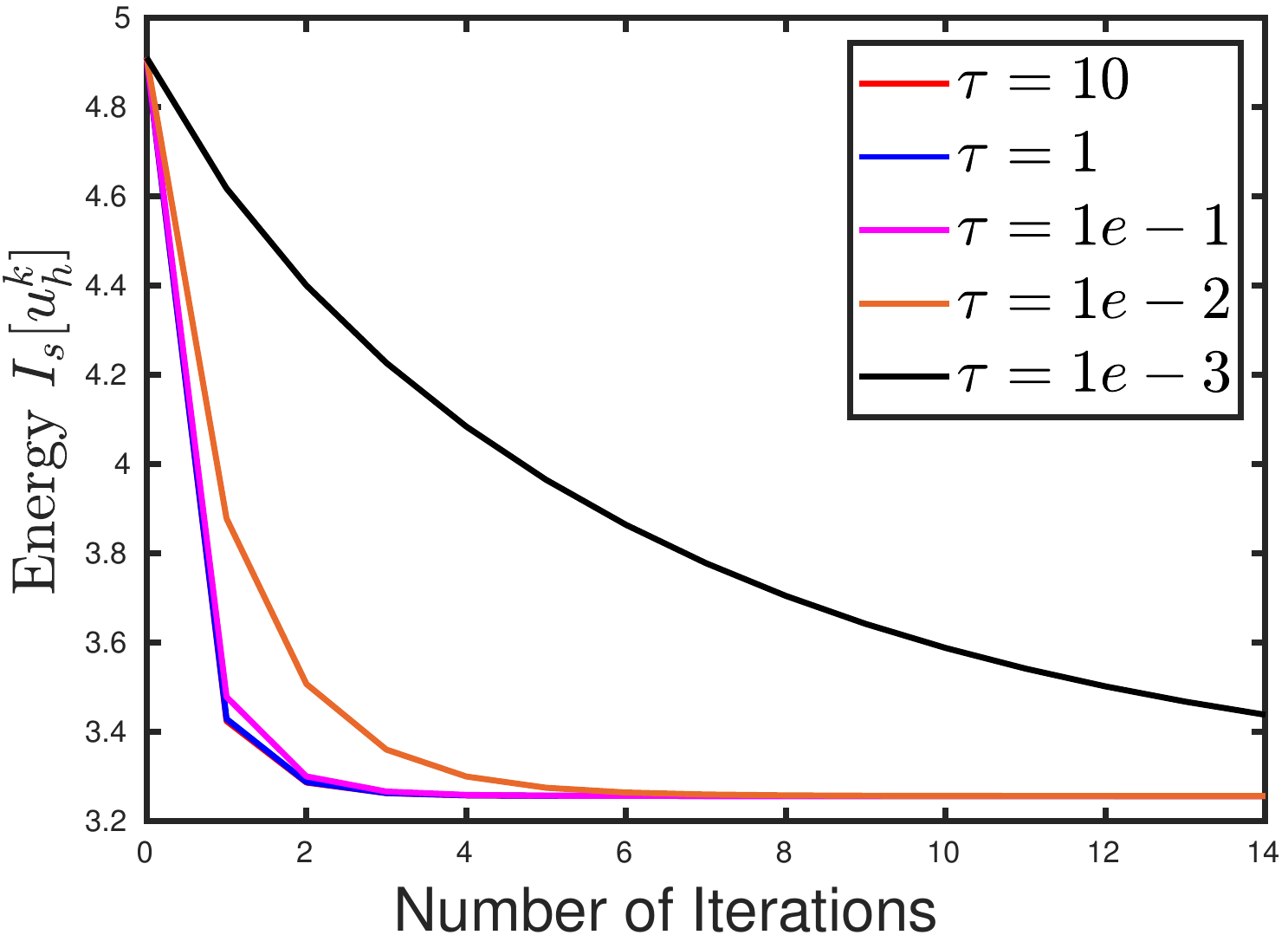}
	\end{center}
	\caption{\small Gradient flow for $s=0.25$ for different choices of $\tau$. Left: norm of residual vector $r$ in the iterative process. Right: energy of $I_s[u_h^k]$ in the iterative process.}
	\label{F:GF-properties}
\end{figure}

\begin{figure}[!htb]
	\begin{center}
		\includegraphics[width=0.6\linewidth]{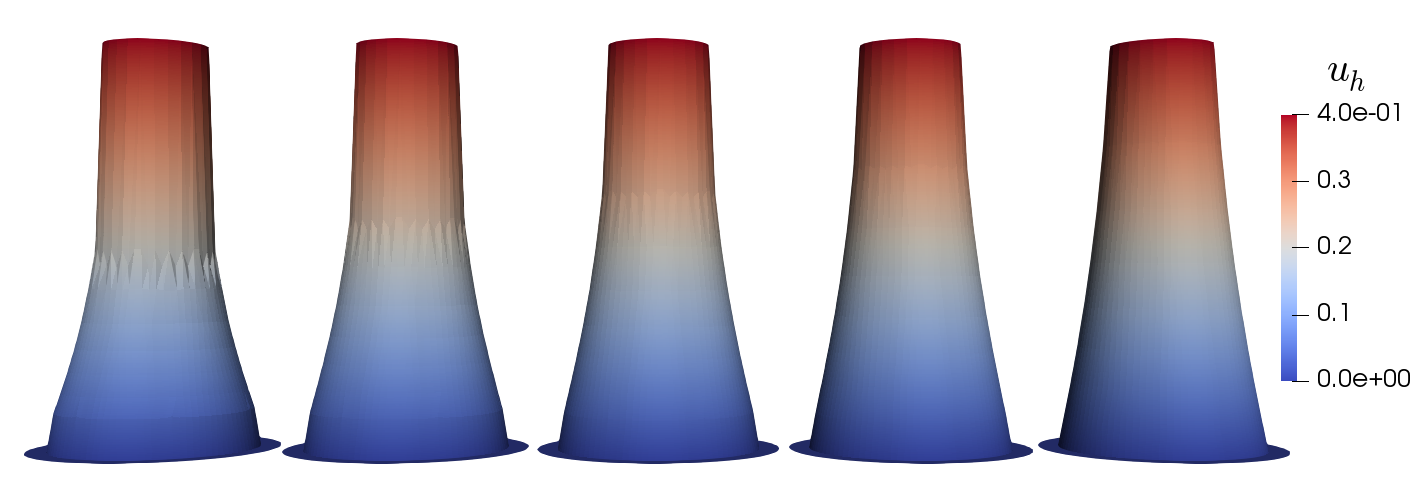}
	\end{center}
	\caption{\small Plot of $u_h$ computed by damped Newton method for $s = 0.05, 0.15, 0.25, 0.35, 0.45$ (from left to right) and $h = 2^{-4}$.}
	\label{F:NMS-Ex3_multis}
\end{figure}

The solution $u_h$ of nonlinear system \eqref{E:WeakForm-discrete} can also be solved using the damped Newton method mentioned in \Cref{SS:schemes}. We choose initial solution $u_h^0 = 0$ and the plots of $u_h$ for several different $s \in (0, 1/2)$ are shown in \Cref{F:NMS-Ex3_multis}. The computed discrete solution for $s = 0.25$ is almost the same as the one computed by gradient flow in \Cref{F:GF-s025}. However, the damped Newton method is more efficient than the gradient flow since we only need $4$ iterations and $243$ seconds compared with $26$ iterations and $800$ seconds when using the gradient flow with $\tau = 1$. 

As shown in the  \Cref{F:NMS-Ex3_multis}, the graph of $u_h$ near $\partial B_{1/2}$ is steeper, and the norm of $\nabla u_h$ larger for smaller $s$, while it becomes smoother, and the norm of $\nabla u_h$ smaller as $s$ increases. This seems to suggest a stickiness phenomenon (see \Cref{R:stickiness}) (stickiness) for small $s$ in this example. We also notice that on the other part of boundary $\partial B_{1}$, the stickiness seems to be small or vanish (i.e. the gap of $u$ on both sides of $\pO$ is small or zero), which is kind of expected since there is no stickiness on $\partial B_{1}$ for the classical case \Cref{R:NMS-classical-annulus}.

Due to the Gamma-convergence result of fractional perimeter in \cite[Theorem 3]{AmbrPhilMart2011Gamma}, $s-$nonlocal minimal graph $u$ converges to the classical minimal graph $u^*$ in $L^1(\Omega)$ as $s \to \frac12^-$. Since we know the analytical solution of classical minimal graph $u^{*}$ in our example, to verify our computation, we could compare the discrete nonlocal minimal graph $u_h$ for $s = 0.499999 \approx \frac12$ with $u^*$. \Cref{F:Newton-multi-h} shows that at least $\Vert u_h - u^* \Vert$ converges for $L^1$ norm to a small number, which indicates the convergence of $u_h$ as $h \to 0$, and a second order convergence rate. Although we could not prove this theoretically, this second order convergence might be due to the fact that $s$ is too close to $0.5$, and the nonlocal graph is almost the same as the classical one. In fact, this $O(h^2)$ convergence rate has been proved for the classical minimal graph problems in $L^1$ norm under proper assumptions for dimension $d=2$ in \cite[Theorem 2]{JoTh75}.

\begin{figure}[!htb]
	\begin{center}
		\includegraphics[width=0.45\linewidth]{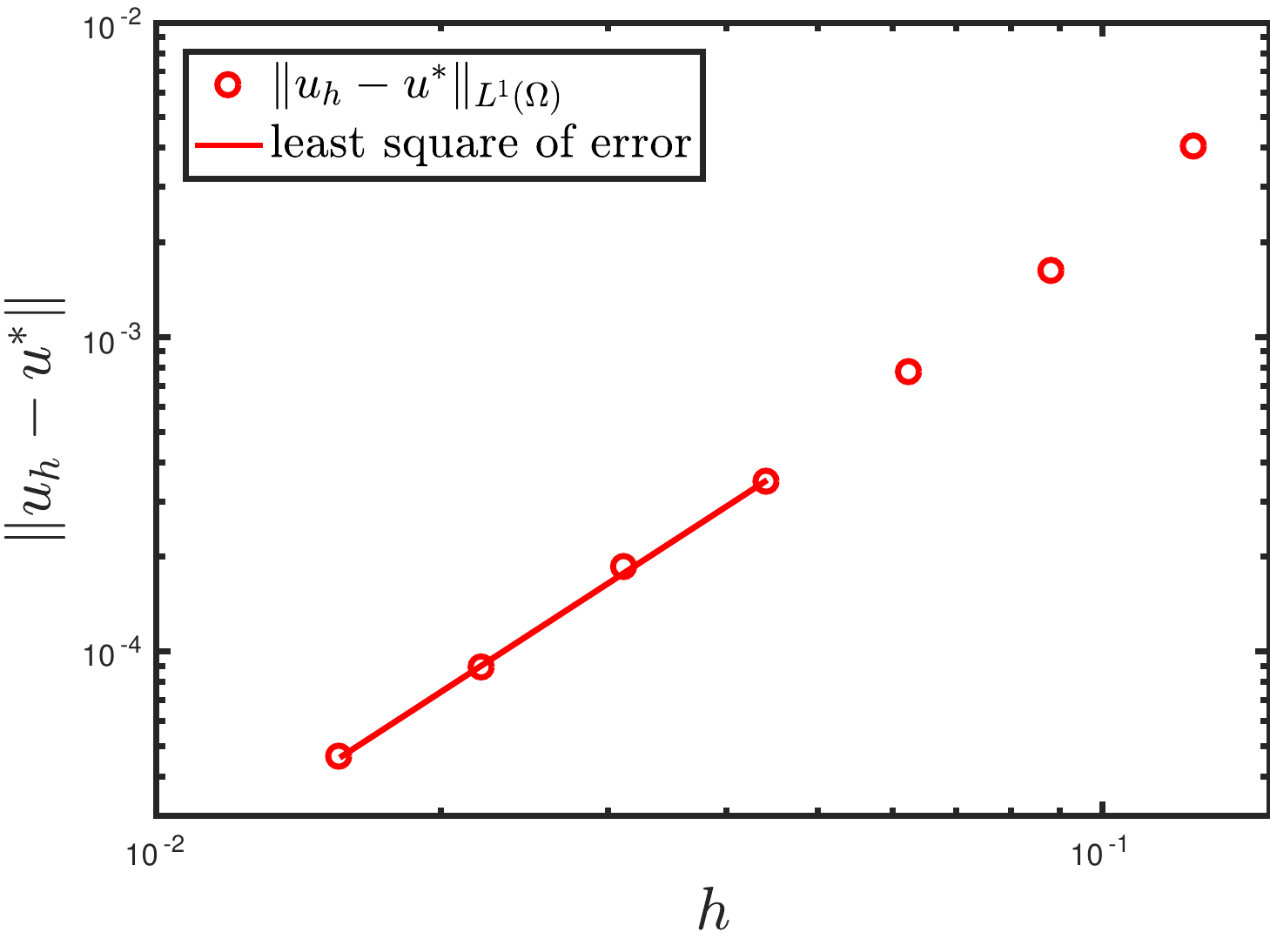}
	\end{center}
	\caption{\small Plot of $\Vert u_h - u^* \Vert_{L^1(\Omega)}$ for different mesh size $h$, where $u_h$ is the discrete solution for $s = 0.499999$ and $u^*$ is the exact solution of classical minimal graph. Least square regression suggests a convergence rate $1.96$, which is close to $O(h^2)$.}
	\label{F:Newton-multi-h}
\end{figure}

%%%%%%%%%%%%%%%%%%%%%%%%%%%%
\appendix
%%%%%%%%%%%%%%%%%%%%%%%%%%%%

%%%%%%%%%%%%%%%%%%%%%%%%%%%%
%%%%%%%%%%%%%%%%%%%%%%%%%%%%
\section{Fractional perimeter and minimal sets} \label{appendix:perimeter}
%%%%%%%%%%%%%%%%%%%%%%%%%%%%
%%%%%%%%%%%%%%%%%%%%%%%%%%%%

The concept of fractional perimeter, that leads
to fractional minimal sets, was introduced in \cite{CaRoSa10}
and has been further developed in
\cite{BucurLombVald18,BucurValdinoci16,DipiFigaPalaVald12Asymptotics,DipiSavinVald16Graph,DipiSavinVald17,DipiSavinVald19,Lomb16Approx,Lombardini-thesis}. Since this justifies the
choice of functional $I_s[u]$ in \eqref{Is}, we review this
rather technical development with emphasis on fractional graphs.

%%%%%%%%%%%%%%%%%%%%%%%%%%%%
\subsection{Fractional perimeters and minimal sets}
%%%%%%%%%%%%%%%%%%%%%%%%%%%%
Here we present the definitions of fractional perimeter and fractional minimal sets and discuss their properties.

\begin{Definition}[$s$-perimeter] \label{def:s-perimeter}
Given a domain $\Omega' \subset \mRdp$ and $s \in (0, 1/2)$, the $s$-perimeter of a set $E \subset \mRdp$
in $\Omega'$ is defined as
\begin{equation}\label{E:NMS-Energy}
    P_s(E,\Omega') := L_s(E \cap \Omega', E^c) + L_s(E \setminus \Omega', \Omega' \setminus E),
\end{equation}
where $E^c := \mRdp \setminus E$ and for any sets $A,B \subset \mRdp$, the interaction between them is defined as
\begin{equation*}
    L_s(A,B) := \iint_{A \times B} \frac{dxdy}{|x-y|^{d+1+2s}}.
\end{equation*}
\end{Definition}

Formally, definition \eqref{E:NMS-Energy} coincides with
\begin{equation*}
    P_s(E,\Omega') = \frac{1}{2} \l( [\chi_E]_{W^{2s}_1(\mRdp)} - [\chi_E]_{W^{2s}_1(\Omega'^c)} \r),
\end{equation*}
where
\begin{equation*}
    [v]_{W^s_p(U)} := \l( \iint_{U \times U} \frac{|v(x)-v(y)|^p}{|x-y|^{d+1+sp}} dxdy \r)^{1/p}
\end{equation*}
is the standard Gagliardo-Aronszajn-Slobodeckij seminorm. 

It is known that, as $s \to \frac12^-$, the scaled $s$-perimeter $P_s(E,\Omega')$ converges to the classical perimeter, see \cite[Theorem 6.0.5]{BucurValdinoci16} and references therein. Indeed, for all $R > 0$ and all sets $E$ with finite perimeter in the ball $B_R$,
\begin{equation*}
\lim_{s \to \frac12^-} \l(\frac{1}{2} - s \r) P_s(E, B_r) = c_{d+1} P(E, B_r),
\end{equation*}
for almost every $r \in (0,R)$, where $c_{d+1}$ is a renormalizing constant and
$P(E,\Omega')$ is defined as
\begin{equation*}
P(E,\Omega') := \sup \l\{ \int_E \div \; \vp \; dx \colon \vp \in C^1_c(\Omega', \mRdp), |\vp|\leq 1 \r\}.
\end{equation*}

On the other hand, the behavior of $P_s$ as $s \to 0$ is investigated in \cite{DipiFigaPalaVald12Asymptotics}, where it is shown
that if $P_{s_0}(E,\Omega') < \infty$ for some $s_0 \in (0, 1/2)$, and the limit 
\begin{equation*}
\alpha(E) := \lim_{s \to 0} \; 2s \int_{E \cap B_1^c} \frac{1}{|y|^{d+1+2s}} dy
\end{equation*}
exists, then
\begin{equation*}
\lim_{s \to 0} \; 2s |\partial B_1| P_s(E, \Omega') = \l( |\partial B_1| - \alpha(E) \r) |E \cap \Omega'|
+ \alpha(E)  \;  |\Omega' \setminus E|.
\end{equation*}
In particular, if $E$ is a bounded set and $P_{s_0}(E, \Omega') < \infty$ for some $s_0$, then $\alpha(E)= 0$ and~$\lim_{s \to 0} 2s P_s(E, \Omega') = |E \cap \Omega'|$.
Therefore, the scaled limit of $P_s(E, \Omega')$ is the measure of $E$ within $\Omega'$ provided $E$ is bounded.

We are now in position to define $s$-minimal sets in $\Omega'$, which are sets that minimize the $s$-fractional perimeter among those that coincide with them outside $\Omega'$. It is noteworthy that this definition does not only involve the behavior of sets in $\overline{\Omega'}$ but rather in the whole space $\mRdp$.

\begin{Definition}[$s$-minimal set]
A set $E$ is $s$-minimal in a open set $\Omega' \subset \mRdp$ if $P_s(E,\Omega')$ is finite and
$P_s(E, \Omega') \leq P_s(F, \Omega')$ among all measurable sets $F \subset \mRdp$ such that
$F \setminus \Omega' = E \setminus \Omega'$. 
The boundary $\partial E$ of a $s$-minimal set $E$ is then called a $s$-minimal surface in $\Omega'$. 
\end{Definition}

Given an open set $\Omega'$ and a fixed set $E_0$, the Dirichlet or Plateau problem for nonlocal minimal surfaces aims to find a $s$-minimal set $E$ such that $E \setminus \Omega' = E_0 \setminus \Omega'$. For a bounded Lipschitz domain $\Omega'$ the existence of solutions to the Plateau problem is established in \cite{CaRoSa10}. 

\begin{Remark}[stickiness] \label{R:stickiness}
A striking difference between nonlocal minimal surface problems and their local counterparts is the emergence of {\em stickiness} phenomena \cite{DipiSavinVald17}: the boundary datum may not be attained continuously. Stickiness is indeed the typical behavior of nonlocal minimal surfaces over bounded domains $\Omega'$. Reference \cite{BucurLombVald18} proves that when $s$ is small and the Dirichlet data occupies, in a suitable sense, less than half the space at infinity, either $s$-minimal sets are empty in $\Omega'$ or they satisfy a density condition.
The latter entails the existence of a $\delta=\delta(s) > 0$ such that for every $x \in \Omega'$ satisfying $B_{\delta}(x) \Subset \Omega'$, it holds that $|E \cap B_{\delta}(x)| > 0$. The recent work \cite{DipiSavinVald19} shows that, in the $1d$ graph setting, there is no intermediate behavior: minimizers either develop jump discontinuities 
or have a H\"older continuous first derivative
across $\partial \Omega'$. 
\end{Remark}

%%%%%%%%%%%%%%%%%%%%%%%%%%%%
\subsection{Fractional minimal graphs}
%%%%%%%%%%%%%%%%%%%%%%%%%%%%

Since we are concerned with graphs, the set $\Omega'=\Omega\times\mR$ is a cylinder and $ E \setminus \Omega'$ is a subgraph. Lombardini points out in \cite[Remark 1.14]{Lomb16Approx} that, in this case, the definition of minimal set as a minimizer of the fractional perimeter is meaningless because $P_s(E,\Omega') = \infty$ for every set $E$.
This issue can be understood by decomposing the fractional perimeter
\[
P_s(E,\Omega') = P_s^L (E,\Omega') + P_s^{NL} (E,\Omega'),
\]
with
\[
\begin{aligned}
P_s^L (E,\Omega')  & = L_s(E \cap \Omega', E^c\cap \Omega') = \frac12 |\chi_E |_{W^{2s}_1(\Omega)}, \\ 
P_s^{NL} (E,\Omega') & =  L_s(E \cap \Omega', E^c \setminus \Omega')+ L_s(E \setminus \Omega', \Omega' \setminus E)  \\ 
& = \iint_{\Omega' \times {\Omega'}^c} \frac{|\chi_E(x) - \chi_E(y)|}{|x-y|^{d+1+2s}} \; dxdy,
\end{aligned}
\]
and realizing that $P_s^{NL}(E,\Omega')$ is trivially infinite independently of $E$. This problem can be avoided by, instead of $s$-minimal sets, seeking for {\em locally} $s$-minimal sets. 

\begin{Definition}[locally $s$-minimal set]
A set $E$ is locally $s$-minimal in $\Omega'$ if it is $s$-minimal in every bounded open subset compactly supported in $\Omega'$.
\end{Definition}

For bounded sets with Lipschitz boundary, the notions of $s$-minimality and local $s$-minimality coincide \cite{Lomb16Approx}. However, as also shown in \cite{Lomb16Approx}, the Plateau problem (in terms of locally $s$-minimal sets) admits solutions even when the domain is unbounded.

\begin{Proposition}[existence of locally $s$-minimal sets]
Let $\Omega' \subset \mRdp$ be an open set and let $E_0 \subset \mRdp$. Then, there exists a set $E \subset \mRdp$ locally $s$-minimal in $\Omega'$, such that $E \setminus \Omega' = E_0 \setminus \Omega'$.
\end{Proposition}

We now consider the minimal graph problem: we assume $\Omega' = \Omega \times \mR$ is a cylinder with $\Omega \subset \mRd$ being a Lipschitz domain, and the Dirichlet datum to be the subgraph of some function $g$ that is bounded and compactly supported (cf. \eqref{E:Def-E0} and \eqref{E:assumptions}).
In this setting, Dipierro, Savin and Valdinoci \cite{DipiSavinVald16Graph} proved that for every locally $s$-minimal set in $\Omega'$ there exists $M_0 > 0$ such that
\begin{equation}\label{E:minimal-priori} 
\Omega \times (-\infty, -M_0) \; \subset \; E \cap \Omega' \; \subset \; \Omega \times (-\infty, M_0).
\end{equation}
As pointed out in \cite[Proposition 2.5.3]{Lombardini-thesis}, a consequence of this estimate is that a set $E$ is locally $s$-minimal in $\Omega' = \Omega \times \mR$ if and only if it is $s$-minimal in $\Omega_M = \Omega \times (-M,M)$ for every $M>M_0$. 
 
Additionally, once the a priori bound \eqref{E:minimal-priori} on the vertical variation of locally $s$-minimal sets is known, it can be shown that minimal sets need to be subgraphs, that is,
\begin{equation}\label{E:Def-E}
E \cap \Omega' = \l\{ (x', x_{d+1}) \colon x_{d+1} < u(x'), \; x' \in \Omega \r\}
\end{equation}
for some function $u$ (cf. \cite[Theorem 4.1.10]{Lombardini-thesis}).
We refer to such a set $E$ as a {\em nonlocal minimal graph} in $\Omega$. 
Thus, as expressed in \Cref{R:solutions}, the Plateau problem for nonlocal minimal graphs consists in finding a function $u: \mRd \to \mR$, with the constraint $u = g$ in ${\Omega}^c$, such that the resulting set $E$ is a locally $s$-minimal set.

%%%%%%%%%%%%%%%%%%%%%%%%%%%%
%%%%%%%%%%%%%%%%%%%%%%%%%%%%
\section{Derivation of the energy \eqref{Is} for graphs: proof of \Cref{prop:perimeter-energy}} \label{appendix:energy}
%%%%%%%%%%%%%%%%%%%%%%%%%%%%
%%%%%%%%%%%%%%%%%%%%%%%%%%%%

In this appendix, we establish the relation between the fractional $s$-perimeter $P_s(E,\Omega')$ of the subgraph of a certain function $u$ given by \eqref{E:Def-E} and the energy functional $I_s[u]$ defined in \eqref{Is}. This will also prove \Cref{prop:perimeter-energy}.

We recall our basic assumptions \eqref{E:assumptions}: $\Omega \subset \mRd$ is a bounded Lipschitz domain and $g \in L^\infty(\Omega^c)$. Given $M > 0$ sufficiently large depending on $s,d,\Omega,g$, we let $\Omega_M = \Omega \times [-M,M]$. We note that, according to \eqref{E:minimal-priori} and \eqref{E:Def-E}, the problem of nonlocal minimal graphs in $\Omega$ reduces to finding a function $u$ in the class
\begin{equation*}
\l\{ u \colon \mRd \to \mR : \  \Vert u \Vert_{L^{\infty}(\Omega)} \leq M, \ u = g \mbox{ in } \Omega^c \r\}
\end{equation*}
such that the set $E := \{ (x', x_{d+1})\in\mRdp: x_{d+1} \leq u(x') \}$ satisfies 
\[
P_s(E, \Omega_M) \le P_s(F, \Omega_M)
\]
for every set $F$ that coincides with $E$ outside $\Omega_M$.
Our goal is to prove \Cref{prop:perimeter-energy}, namely to show that
\[
P_s (E, \Omega_M) = I_s[u] + C(M,d,s,\Omega, g), 
\]
where $I_s$ is given \eqref{Is} and \eqref{E:NMS-Energy-Graph} and reads
\[
I_s[u] = \iint_{Q_{\Omega}} F_s\l(\frac{u(x)-u(y)}{|x-y|}\r) \frac{1}{|x-y|^{d+2s-1}} \;dxdy.
\]

This identity will follow by elementary arguments, inspired in Lombardini \cite{Lomb16Approx};
further details can be found in \cite[Chapter 4]{Lombardini-thesis}. 
%basically due to Lombardini\cite[Chapter 4]{Lombardini-thesis}. 
Definition \eqref{E:NMS-Energy} yields
\begin{equation} \label{E:s-perimeter}
    P_s(E, \Omega_M) = L_s(E \cap \Omega_M, E^c) + 
    L_s(E \setminus \Omega_M, E^c \cap \Omega_M).
\end{equation}
For the first term $I$ on the right hand side above, we write $I=I_1+I_2$ where
\begin{equation*}
    I := L_s(E \cap \Omega_M, E^c) =
    \iint_{\Omega \times \mRd} dxdy \int_{-M}^{u(x)} dt
    \int_{u(y)}^{\infty} \frac{dr}{\l((t-r)^2+|x-y|^2\r)^{(d+1+2s)/2}}
\end{equation*}
and
\begin{align*}
I_1 & := \iint_{\Omega \times \Omega} dxdy \int_{-M-u(y)}^{u(x)-u(y)} dt
    \int_{-t}^{\infty} \frac{dr}{\l(r^2+|x-y|^2\r)^{(d+1+2s)/2}}, \\
I_2 & := \iint_{\Omega \times {\Omega}^c} dxdy 
    \int_{-M-u(y)}^{u(x)-u(y)} dt
    \int_{-t}^{\infty} \frac{dr}{\l(r^2+|x-y|^2\r)^{(d+1+2s)/2}}.
\end{align*}
Recalling that $\Omega' = \Omega \times \mR$, the second term $II$ in \eqref{E:s-perimeter} can be split as
\begin{equation*}
II := L_s(E \setminus \Omega_M, E^c \cap \Omega_M) = II_1 + II_2,
\end{equation*}
where
\begin{equation}\label{eq:second_term}
II_1 := L_s( (E\cap \Omega') \setminus \Omega_M, E^c \cap \Omega_M),
\quad
II_2 := L_s(E \setminus \Omega', E^c \cap \Omega_M).
\end{equation}
Applying Fubini's Theorem and the change of variables $(r,t) = (-\tilde{r}-\tilde{t}, -\tilde{r}-M)$, we obtain
\begin{equation*}
\begin{aligned}
II_1 &= \iint_{\Omega \times \Omega} dxdy \int_{-\infty}^{-M} d\tilde{t}
    \int_{u(y)}^{M} \frac{d\tilde{r}}{\l( (\tilde{t}-\tilde{r})^2+|x-y|^2 \r)^{(d+1+2s)/2}} \\
    &= \iint_{\Omega \times \Omega} dxdy \int_{-2M}^{-u(y)-M} dt
    \int_{-t}^{\infty} \frac{dr}{\l(r^2+|x-y|^2\r)^{(d+1+2s)/2}}. \\    
\end{aligned}
\end{equation*}
Therefore, we have
\begin{equation*}
\begin{aligned}
I_1 + II_1 &= \iint_{\Omega \times \Omega} dxdy \int_{-2M}^{u(x)-u(y)} dt
    \int_{-t}^{\infty} \frac{dr}{(r^2+|x-y|^2)^{(d+1+2s)/2}} \\
&= \iint_{\Omega \times \Omega} \frac{dxdy}{|x-y|^{d-1+2s}} 
\int_{\frac{-2M}{|x-y|}}^{\frac{u(x)-u(y)}{|x-y|} } dt 
\int_{-t}^{\infty} \frac{dr}{(r^2+1)^{(d+1+2s)/2}},
\end{aligned}
\end{equation*}
and using the symmetry in $(x,y)$ of the integral over $\Omega \times \Omega$, we arrive at
\begin{equation*}
\begin{aligned}
I_1 + II_1 = \frac{1}{2} \iint_{\Omega \times \Omega} \frac{dxdy}{|x-y|^{d-1+2s}} 
\Bigg( & \int_{\frac{-2M}{|x-y|}}^{\frac{u(x)-u(y)}{|x-y|} } dt 
\int_{-t}^{\infty} \frac{dr}{(r^2+1)^{(d+1+2s)/2}} \\
& \quad +
\int_{\frac{-2M}{|x-y|}}^{\frac{u(y)-u(x)}{|x-y|} } dt 
\int_{-t}^{\infty} \frac{dr}{(r^2+1)^{(d+1+2s)/2}}
\Bigg).
\end{aligned}
\end{equation*}

Next, the splitting
\[ \begin{aligned}
\int_{\frac{-2M}{|x-y|}}^{\frac{u(x)-u(y)}{|x-y|} } & dt  \int_{-t}^{\infty} \frac{dr}{(r^2+1)^{(d+1+2s)/2}}  \\ 
& = \int_{\frac{-2M}{|x-y|}}^{0} dt 
\int_{-t}^{\infty} \frac{dr}{(r^2+1)^{(d+1+2s)/2}} - \int_{0}^{\frac{u(y)-u(x)}{|x-y|}} dt 
\int_{t}^{\infty} \frac{dr}{(r^2+1)^{(d+1+2s)/2}}
\end{aligned} \]
gives
\begin{equation*}
\begin{aligned}
& \int_{\frac{-2M}{|x-y|}}^{\frac{u(x)-u(y)}{|x-y|} } dt 
\int_{-t}^{\infty} \frac{dr}{(r^2+1)^{(d+1+2s)/2}} +
\int_{\frac{-2M}{|x-y|}}^{\frac{u(y)-u(x)}{|x-y|} } dt 
\int_{-t}^{\infty} \frac{dr}{(r^2+1)^{(d+1+2s)/2}} \\
&  = 2\int_{\frac{-2M}{|x-y|}}^{0} dt 
\int_{-t}^{\infty} \frac{dr}{(r^2+1)^{(d+1+2s)/2}} +
\int_{0}^{\frac{u(y)-u(x)}{|x-y|}} dt 
\int_{-t}^{t} \frac{dr}{(r^2+1)^{(d+1+2s)/2}}.
\end{aligned}
\end{equation*}

Thus, collecting the estimates above and recalling definition \eqref{E:def_Fs}, we deduce
\begin{equation*}
\begin{aligned}
I_1 + II_1 &= C_1 + \iint_{\Omega \times \Omega} 
F_s\l(\frac{u(x)-u(y)}{|x-y|}\r) \frac{1}{|x-y|^{d-1+2s}} dxdy,
\end{aligned}
\end{equation*}
where 
\[
C_1 := \iint_{\Omega \times \Omega} dx dy \int_{\frac{-2M}{|x-y|}}^{0} dt 
\int_{-t}^{\infty} \frac{dr}{(r^2+1)^{(d+1+2s)/2}} 
\]
is a finite number that only depends on $M,d,s,\Omega$.
The finiteness of $C_1$ is due to the boundedness of $\Omega$ and the bound
\begin{equation*}
\begin{aligned}
\int_{0}^{\frac{2M}{|x-y|}} dt  &
\int_{t}^{\infty} \frac{dr}{(r^2+1)^{(d+1+2s)/2}} \\
 \leq & \int_{0}^{\infty} dt 
\int_{t}^{\infty} \frac{dr}{(r^2+1)^{(d+1+2s)/2}} 
= \int_{0}^{\infty} \frac{r \;dr}{(r^2+1)^{(d+1+2s)/2}} < \infty.
\end{aligned}
\end{equation*}

Applying the change of variables $(t,r) = (-\tilde{r}+u(y), \tilde{r}-\tilde{t})$, the term $II_2 = L_s(E \setminus \Omega, E^c \cap \Omega_M)$ from \eqref{eq:second_term} can be expressed as
\begin{equation*}
\begin{aligned}
II_2 &= \iint_{{\Omega}^c \times \Omega} dxdy \int_{-\infty}^{u(x)} d\tilde{t}
    \int_{u(y)}^{M} \frac{d\tilde{r}}{\l((\tilde{t}-\tilde{r})^2+|x-y|^2\r)^{(d+1+2s)/2}} \\
&= \iint_{\Omega \times{\Omega}^c} dxdy \int_{-\infty}^{u(y)} d\tilde{t}
    \int_{u(x)}^{M} \frac{d\tilde{r}}{\l((\tilde{t}-\tilde{r})^2+|x-y|^2\r)^{(d+1+2s)/2}} \\
&= \iint_{\Omega \times{\Omega}^c} dxdy \int_{-M+u(y)}^{u(y)-u(x)} dt
    \int_{-t}^{\infty} \frac{dr}{(r^2+|x-y|^2)^{(d+1+2s)/2}}.
\end{aligned}
\end{equation*}

We next combine $I_2$ and $II_2$ to obtain
\begin{equation*}
\begin{array}{l r}
I_2 + II_2 &= \displaystyle{ \iint_{\Omega \times {\Omega}^c} 
\frac{dxdy}{|x-y|^{d-1+2s}} \Bigg(
\int_{\frac{-M-u(y)}{|x-y|}}^{\frac{u(x)-u(y)}{|x-y|} } dt 
\int_{-t}^{\infty} \frac{dr}{(r^2+1)^{(d+1+2s)/2}} }\\
& + \displaystyle{ \int_{\frac{-M+u(y)}{|x-y|}}^{\frac{u(y)-u(x)}{|x-y|} } dt 
\int_{-t}^{\infty} \frac{dr}{(r^2+1)^{(d+1+2s)/2}} \Bigg)} \\
& = \displaystyle{ \iint_{\Omega \times {\Omega}^c}  \frac{dxdy}{|x-y|^{d-1+2s}} \Bigg(
\int_{\frac{-M-u(y)}{|x-y|}}^{\frac{u(x)-u(y)}{|x-y|} } dt 
\int_{-t}^{0} \frac{dr}{(r^2+1)^{(d+1+2s)/2}}} \\
& + \displaystyle{  \int_{\frac{-M+u(y)}{|x-y|}}^{\frac{u(y)-u(x)}{|x-y|} } dt
\int_{-t}^{0} \frac{dr}{(r^2+1)^{(d+1+2s)/2}} + \frac{2M}{|x-y|} K \Bigg),} 
\end{array}
\end{equation*}
where $K = \int_0^{\infty}(r^2+1)^{-(d+1+2s)/2} dr$.
Therefore, recalling once again \eqref{E:def_Fs}, we deduce
\begin{equation*}
\begin{array}{lr}
I_2 + II_2 & = \displaystyle{\iint_{\Omega \times {\Omega}^c} 
\frac{dxdy}{|x-y|^{d-1+2s}} \bigg(
2 F_s\l(\frac{u(x)-u(y)}{|x-y|}\r)} \\
& - \displaystyle{F_s\l(\frac{-M-g(y)}{|x-y|}\r)}- F_s\l(\frac{M-g(y)}{|x-y|} \r) \bigg) + C_2,
\end{array} 
\end{equation*}
with $C_2 = 2M K \iint_{\Omega \times {\Omega}^c} |x-y|^{-(d+2s)} dxdy < \infty$, because $\Omega$ is bounded Lipschitz. Additionally, note that because $g \in L^{\infty}({\Omega}^c)$, we have
\begin{equation*}
     \iint_{\Omega \times {\Omega}^c} 
\l( F_s\l(\frac{-M-g(y)}{|x-y|}\r) 
+ F_s\l(\frac{M-g(y)}{|x-y|} \r) \r)\frac{dxdy}{|x-y|^{d-1+2s}}  < \infty.
\end{equation*}

Since $P_s(E, \Omega_M) = I_1 + I_2 + II_1 + II_2$, collecting the estimates above
yields
\begin{equation*}
P_s(E, \Omega_M) 
= \iint_{Q_{\Omega}} F_s\l(\frac{u(x)-u(y)}{|x-y|}\r) \frac{1}{|x-y|^{d+2s-1}} \;dxdy + C(M,d,s,\Omega,g). 
\end{equation*}
This finishes the proof of \Cref{prop:perimeter-energy}, and shows that the function $u$, whose subgraph solves the nonlocal Plateau problem in $\Omega'$, minimizes the energy \eqref{E:NMS-Energy-Graph}.

%%%%%%%%%%%%%%%%%%%%%%%%%%%%
%%%%%%%%%%%%%%%%%%%%%%%%%%%%
\bibliographystyle{amsplain}
\bibliography{NMS}

\providecommand{\bysame}{\leavevmode\hbox to3em{\hrulefill}\thinspace}
\providecommand{\MR}{\relax\ifhmode\unskip\space\fi MR }
% \MRhref is called by the amsart/book/proc definition of \MR.
\providecommand{\MRhref}[2]{%
  \href{http://www.ams.org/mathscinet-getitem?mr=#1}{#2}
}
\providecommand{\href}[2]{#2}
\begin{thebibliography}{10}

\bibitem{AcosBersBort2017short}
G.~Acosta, F.M. Bersetche, and J.P. Borthagaray, \emph{A short {FE}
  implementation for a 2d homogeneous {D}irichlet problem of a fractional
  {L}aplacian}, Comput. Math. Appl. \textbf{74} (2017), no.~4, 784--816.

\bibitem{AcosBort2017fractional}
G.~Acosta and J.P. Borthagaray, \emph{A fractional {L}aplace equation:
  regularity of solutions and finite element approximations}, SIAM J. Numer.
  Anal. \textbf{55} (2017), no.~2, 472--495.

\bibitem{AcBoHe18}
G.~Acosta, J.P. Borthagaray, and N.~Heuer, \emph{Finite element approximations
  of the nonhomogeneous fractional dirichlet problem}, IMA J. Numer. Anal.
  (2018).

\bibitem{AiGl18}
M.~Ainsworth and C.~Glusa, \emph{Towards an efficient finite element method for
  the integral fractional {L}aplacian on polygonal domains}, pp.~17--57,
  Springer International Publishing, Cham, 2018.

\bibitem{AmbrPhilMart2011Gamma}
L.~Ambrosio, G.~De~Philippis, and L.~Martinazzi, \emph{Gamma-convergence of
  nonlocal perimeter functionals}, Manuscripta Math. \textbf{134} (2011),
  no.~3, 377--403.

\bibitem{AntilKhatriWarma18}
H.~Antil, R.~Khatri, and M.~Warma, \emph{External optimal control of nonlocal
  {PDEs}}, arXiv preprint arXiv:1811.04515, 2018.

\bibitem{BaMoNo04}
E.~B\"{a}nsch, P.~Morin, and R.H. Nochetto, \emph{Surface diffusion of graphs:
  variational formulation, error analysis, and simulation}, SIAM J. Numer.
  Anal. \textbf{42} (2004), no.~2, 773--799. \MR{2084235}

\bibitem{Barrios2014bootstrap}
B.~Barrios, A.~Figalli, and E.~Valdinoci, \emph{Bootstrap regularity for
  integro-differential operators, and its application to nonlocal minimal
  surfaces}, Ann. Sc. Norm. Super. Pisa Cl. Sci. \textbf{13} (2014), no.~3,
  609--639.

\bibitem{BBNOS18}
A.~Bonito, J.P. Borthagaray, R.H. Nochetto, E.~Ot{\'a}rola, and A.J. Salgado,
  \emph{Numerical methods for fractional diffusion}, Comput. Vis. Sci.
  \textbf{19} (2018), no.~5, 19--46.

\bibitem{BoLePa17}
A.~Bonito, W.~Lei, and J.E. Pasciak, \emph{Numerical approximation of the
  integral fractional {L}aplacian}, Numer. Mat. (2019).

\bibitem{BoLeSa18}
A.~Bonito, W.~Lei, and A.J. Salgado, \emph{Finite element approximation of an
  obstacle problem for a class of integro-differential operators}, arXiv
  preprint arXiv:1808.01576, 2018.

\bibitem{BoLiNo19computational}
J.P. Borthagaray, W.~Li, and R.H. Nochetto, \emph{Finite element
  discretizations for nonlocal minimal graphs: computational aspects}, In
  preparation, 2019.

\bibitem{BoNoSa18}
J.P. Borthagaray, R.H. Nochetto, and A.J. Salgado, \emph{{Weighted Sobolev
  regularity and rate of approximation of the obstacle problem for the integral
  fractional Laplacian}}, arXiv preprint arXiv:1806.08048, 2018.

\bibitem{BourBrezMiro2001another}
J.~Bourgain, H.~Brezis, and P.~Mironescu, \emph{Another look at {S}obolev
  spaces}, Optimal Control and Partial Differential Equations, 2001,
  pp.~439--455.

\bibitem{BucurLombVald18}
C.~Bucur, L.~Lombardini, and E.~Valdinoci, \emph{Complete stickiness of
  nonlocal minimal surfaces for small values of the fractional parameter}, Ann.
  Inst. H. Poincar\'e Anal. Non Lin\'eaire \textbf{36} (2019), no.~3, 655--703.

\bibitem{BucurValdinoci16}
C.~Bucur and E.~Valdinoci, \emph{Nonlocal diffusion and applications}, vol.~20,
  Springer, 2016.

\bibitem{BuGu18}
O.~Burkovska and M.~Gunzburger, \emph{Regularity and approximation analyses of
  nonlocal variational equality and inequality problems}, arXiv preprint
  arXiv:1804.10282, 2018.

\bibitem{CabreCozzi2017gradient}
X.~Cabr{\'e} and M.~Cozzi, \emph{A gradient estimate for nonlocal minimal
  graphs}, Duke Math. J. (2019).

\bibitem{CaRoSa10}
L.~Caffarelli, J.-M. Roquejoffre, and O.~Savin, \emph{Nonlocal minimal
  surfaces}, Comm. Pure Appl. Math. \textbf{63} (2010), no.~9, 1111--1144.

\bibitem{CaSaVa15}
L.~Caffarelli, O.~Savin, and E.~Valdinoci, \emph{Minimization of a fractional
  perimeter-{D}irichlet integral functional}, Ann. Inst. H. Poincar\'e Anal.
  Non Lin\'eaire, vol.~32, Elsevier, 2015, pp.~901--924.

\bibitem{CaSo10}
L.~A Caffarelli and P.E. Souganidis, \emph{Convergence of nonlocal threshold
  dynamics approximations to front propagation}, Arch. Ration. Mech. Anal.
  \textbf{195} (2010), no.~1, 1--23.

\bibitem{ChaMorPon12}
A.~Chambolle, M.~Morini, and M.~Ponsiglione, \emph{A nonlocal mean curvature
  flow and its semi-implicit time-discrete approximation}, SIAM J. Math. Anal.
  \textbf{44} (2012), no.~6, 4048--4077.

\bibitem{ChaMorPon15}
\bysame, \emph{Nonlocal curvature flows}, Arch. Ration. Mech. Anal.
  \textbf{218} (2015), no.~3, 1263--1329.

\bibitem{Ciarlet02}
P.~G. Ciarlet, \emph{The finite element method for elliptic problems}, SIAM,
  2002.

\bibitem{CoFi17}
M.~Cozzi and A.~Figalli, \emph{Regularity theory for local and nonlocal minimal
  surfaces: an overview}, Nonlocal and Nonlinear Diffusions and Interactions:
  New Methods and Directions, Springer, 2017, pp.~117--158.

\bibitem{DeckelnickDziuk00}
K.~Deckelnick and G.~Dziuk, \emph{Error estimates for a semi-implicit fully
  discrete finite element scheme for the mean curvature flow of graphs},
  Interfaces Free Bound. \textbf{2} (2000), no.~4, 341--359.

\bibitem{DeDzEl05}
K.~Deckelnick, G.~Dziuk, and C.M. Elliott, \emph{Computation of geometric
  partial differential equations and mean curvature flow}, Acta Numer.
  \textbf{14} (2005), 139--232.

\bibitem{DEliaGunzburger}
M.~D'Elia and M.~Gunzburger, \emph{The fractional {L}aplacian operator on
  bounded domains as a special case of the nonlocal diffusion operator},
  Comput. Math. Appl. \textbf{66} (2013), no.~7, 1245 -- 1260.

\bibitem{DipiFigaPalaVald12Asymptotics}
S.~Dipierro, A.~Figalli, G.~Palatucci, and E.~Valdinoci, \emph{Asymptotics of
  the $s$-perimeter as $s \searrow 0$}, Discrete Contin. Dyn. Syst. \textbf{33}
  (2013), no.~7, 2777--2790.

\bibitem{DipiSavinVald15}
S.~Dipierro, O.~Savin, and E.~Valdinoci, \emph{A nonlocal free boundary
  problem}, SIAM J. Math. Anal. \textbf{47} (2015), no.~6, 4559--4605.

\bibitem{DipiSavinVald16Graph}
\bysame, \emph{Graph properties for nonlocal minimal surfaces}, Calc. Var.
  Partial Differential Equations \textbf{55} (2016), no.~4, 86.

\bibitem{DipiSavinVald17}
\bysame, \emph{Boundary behavior of nonlocal minimal surfaces}, J. Funct. Anal.
  \textbf{272} (2017), no.~5, 1791--1851.

\bibitem{DipiSavinVald19}
\bysame, \emph{{Nonlocal minimal graphs in the plane are generically sticky}},
  arXiv preprint arXiv:1904.05393, 2019.

\bibitem{DipiVald18}
S.~Dipierro and E.~Valdinoci, \emph{(non)local and (non)linear free boundary
  problems}, Discrete Contin. Dyn. Syst. Ser. S \textbf{11} (2018), 465.

\bibitem{Faermann2}
B.~Faermann, \emph{Localization of the {A}ronszajn-{S}lobodeckij norm and
  application to adaptive boundary element methods. {I}. {T}he two-dimensional
  case}, IMA J. Numer. Anal. \textbf{20} (2000), no.~2, 203--234.

\bibitem{Faermann}
\bysame, \emph{Localization of the {A}ronszajn-{S}lobodeckij norm and
  application to adaptive boundary element methods. {II}. {T}he
  three-dimensional case}, Numer. Math. \textbf{92} (2002), no.~3, 467--499.

\bibitem{FierroVeeser03}
F.~Fierro and A.~Veeser, \emph{On the a posteriori error analysis for equations
  of prescribed mean curvature}, Math. Comp. \textbf{72} (2003), no.~244,
  1611--1634.

\bibitem{Figalli2017regularity}
A.~Figalli and E.~Valdinoci, \emph{Regularity and {B}ernstein-type results for
  nonlocal minimal surfaces}, J. Reine Angew. Math. \textbf{2017} (2017),
  no.~729, 263--273.

\bibitem{Grisvard}
P.~Grisvard, \emph{Elliptic problems in nonsmooth domains}, Monographs and
  Studies in Mathematics, vol.~24, Pitman (Advanced Publishing Program),
  Boston, MA, 1985.

\bibitem{Imbert09}
Cyril Imbert, \emph{Level set approach for fractional mean curvature flows},
  Interfaces Free Bound. \textbf{11} (2009), no.~1, 153--176.

\bibitem{JoTh75}
C.~Johnson and V.~Thom{\'e}e, \emph{Error estimates for a finite element
  approximation of a minimal surface}, Math. Comp. \textbf{29} (1975), no.~130,
  343--349.

\bibitem{Lomb16Approx}
L.~Lombardini, \emph{Approximation of sets of finite fractional perimeter by
  smooth sets and comparison of local and global $ s $-minimal surfaces}, arXiv
  preprint arXiv:1612.08237, 2016.

\bibitem{Lombardini-thesis}
\bysame, \emph{Minimization problems involving nonlocal functionals: Nonlocal
  minimal surfaces and a free boundary problem}, Ph.D. thesis, Universita degli
  Studi di Milano and Universite de Picardie Jules Verne, 2018.

\bibitem{Modica87}
L.~Modica, \emph{The gradient theory of phase transitions and the minimal
  interface criterion}, Arch. Rational Mech. Anal. \textbf{98} (1987), no.~2,
  123--142.

\bibitem{ModicaMortola77}
L.~Modica and S.~Mortola, \emph{Un esempio di {$\Gamma$}-convergenza}, Boll.
  Un. Mat. Ital. B (5) \textbf{14} (1977), no.~1, 285--299.

\bibitem{Rannacher}
R.~Rannacher, \emph{Some asymptotic error estimates for finite element
  approximation of minimal surfaces}, Rev. Fran\c caise Automat. Informat.
  Recherche Op\'erationnelle S\'er. Rouge Anal. Num\'er. \textbf{11} (1977),
  181--196.

\bibitem{SaVa12Gamma}
O.~Savin and E.~Valdinoci, \emph{{$\Gamma$-convergence for nonlocal phase
  transitions}}, Ann. Inst. H. Poincar\'{e} Anal. Non Lin\'{e}aire \textbf{29}
  (2012), no.~4, 479--500.

\end{thebibliography}
%%%%%%%%%%%%%%%%%%%%%%%%%%%%
%%%%%%%%%%%%%%%%%%%%%%%%%%%%
\end{document}